\documentclass[10pt]{amsart}
\usepackage{amssymb}
\usepackage{amsmath}
\usepackage{amsthm}
\usepackage{latexsym}
\usepackage{graphicx}
\usepackage{hyperref}
\usepackage{comment}

\newcommand{\Z}{{\mathbb Z}}
\newcommand{\R}{{\mathbb R}}

\newcommand{\T}{{\mathbb T}}
\newcommand{\N}{{\mathbb N}}

\newtheorem{lemma}{Lemma}[section]
\newtheorem{theorem}[lemma]{Theorem}

\newtheorem{proposition}[lemma]{Proposition}
\newtheorem{corollary}[lemma]{Corollary}
\newtheorem{definition}[lemma]{Definition}

\newcommand{\nn}{\nonumber}
\newcommand{\be}{\begin{equation}}
\newcommand{\ee}{\end{equation}}
\newcommand{\ul}{\underline}
\newcommand{\ol}{\overline}
\newcommand{\ti}{\tilde}

\newcommand{\spr}[2]{\langle #1 , #2 \rangle}

\newcommand{\E}{\mathrm{e}}
\newcommand{\I}{\mathrm{i}}

\newcommand{\tr}{\mathrm{tr}}
\newcommand{\im}{\mathrm{Im}}

\DeclareMathOperator{\dist}{dist}

\newcommand{\eps}{\varepsilon}

\newcommand{\lam}{\lambda}

\numberwithin{equation}{section}

\begin{document}

\title[Positivity of Lyapunov exponents]{Multiscale Analysis for Ergodic Schr\"odinger operators and
 positivity of Lyapunov exponents}

\author[H.\ Kr\"uger]{Helge Kr\"uger}
\address{Department of Mathematics, Rice University, Houston, TX~77005, USA}
\email{\href{mailto:helge.krueger@rice.edu}{helge.krueger@rice.edu}}
\urladdr{\href{http://math.rice.edu/~hk7/}{http://math.rice.edu/\~{}hk7/}}

\thanks{H.\ K.\ was supported by NSF grant DMS--0800100.}

\date{\today}


\begin{abstract}
 A variant of multiscale analysis for ergodic Schr\"odinger operators
 is developed. This enables us to prove positivity of Lyapunov exponents
 given initial scale estimates and an initial Wegner estimate.
 This is then applied to high dimensional skew-shifts at small coupling,
 where initial conditions are checked using the Pastur--Figotin formalism.
 Furthermore, it is shown that for potentials generated by the doubling
 map one has positive Lyapunov exponent except in a superpolynomially small
 set.
\end{abstract}

\maketitle

\section{Introduction}

The discrete one dimensional Schr\"odinger operator is one of the simplest
models in quantum mechanics. It describes the motion of a particle in an
one dimensional medium. Of particular interest is the case of ergodic
potentials, where the potential $V$ is given by
\be
 V_{\omega}(n) = f(T^n \omega)
\ee
for $(\Omega,\mu)$ a probability space, $f:\Omega\to\R$ a real-valued and bounded function,
$T:\Omega\to\Omega$ an invertible and ergodic transformation, and $\omega\in\Omega$.
Then the Schr\"odinger operator is given by
\be\begin{split}
 H_{\omega}: \ell^2(\Z)&\to \ell^2(\Z) \\
 H_{\omega} u(n) & = u(n+1) + u(n-1) + V_{\omega}(n) u(n).
\end{split}\ee
If one considers $H_{\omega} u = E u$ as a formal difference equation,
the Lyapunov exponent $L(E)$ describes the exponential behavior of
solutions for almost every $\omega$. It follows from Kotani theory,
that the essential closure of the set $\{E:\quad L(E) = 0\}$
is the absolutely continuous spectrum of $H_{\omega}$ for almost
every $\omega$. Furthermore, in the presence of uniform lower bounds
$L(E) \geq \gamma > 0$, there has been a considerable development of
machinery around the turn of the last millennium that implies localization.

In \cite{schl}, Schlag has posed the following two open problems (and others)
\begin{enumerate}
 \item Positivity of the Lyapunov exponent for small disorders for the skew-shift ($T(x,y) = (x + \alpha, y + x) \pmod{1}$, $\alpha\notin\mathbb{Q}$).
 \item Positivity of the Lyapunov exponent and Anderson localization for all positive disorders
  with $T x = 2 x \pmod 1$.
\end{enumerate}
These two problems have also been noted by Bourgain in \cite{bbook},
Chulaevsky and Spencer \cite{chsp}, Damanik \cite{da1}, and Goldstein and Schlag \cite{gs3}.

My original goal was to make a progress on the first problem, and
it turned out that I needed to enlarge the torus on which
the skew-shift is acting to obtain results, see Theorem~\ref{thm2} and \ref{thm3}.
However, I was fortunate enough that my methods also applied
to the second problem, for which I can extend the range, where
positive Lyapunov exponent holds from a region of small disorders
to also include the region of large disorders, see Theorem~\ref{thm1}.

Most of the methods developed in this paper are independent
of the underlying ergodic transformation, and I will present
these in Section~\ref{sec:results}. However, since the already
mentioned transformations are of special importance, I have
decided to state the results for them in the next section,
while reviewing parts of the current knowledge on them.

%
%
%

\section{The doubling map and the skew-shift}

The plan for this section,
is as follows, we will first review the current knowledge for the doubling
map, and then state our new result, and then repeat this for the skew-shift.
I wish to remark here, that random and quasi-periodic Schr\"odinger operators
are reasonably well understood (see e.g. \cite{bbook}, \cite{dkkkr}).
In order to keep this paper at a reasonable length, I have decided
not to discuss these two examples.

One of the most prototypical examples of a deterministic map, which behaves
close to random, is the doubling map. Let $\Omega = \mathbb{T} \cong [0,1)$
be the unit circle, and introduce $T: \Omega\to\Omega$ by
\be
 T\omega = 2 \omega \pmod{1}.
\ee
It is well-known, that $T$ is ergodic with respect to the Lebesgue measure.
Furthermore, if we consider the dyadic expansion of $\omega$
$$
 \omega = \sum_{j = 1}^{\infty} \frac{\omega_j}{2^j},\quad \omega_j \in \{0,1\}
$$
then the action of $T$ is conjugated to the left shift
$$
 \{\omega_j\}_{j = 1}^{\infty} \mapsto \{\omega_{j+1}\}_{j = 1}^{\infty}
$$
on the space $\{0,1\}^{\N}$. Let $f: \Omega\to\R$ be a continuous
function, and introduce for $\omega\in\Omega$ the potential
\be
 V_{\omega}(n) = f(T^n \omega),\quad n \geq 0.
\ee
Denote by $\Delta$ the discrete Laplacian. It seems natural to expect the Schr\"odinger
operator
\be
 H_{\omega} = \Delta + V_{\omega}
\ee
to behave like a random Schr\"odinger operator (see \cite{bs}),
and in particular have positive Lyapunov exponent for all energies.
In order to state results, it turns out convenient to introduce
a coupling constant $\lambda > 0$, so that
\be
 H_{\omega, \lambda} = \Delta + \lambda V_{\omega}
\ee
Denote by $L_{\lambda}(E)$ the Lyapunov exponent of this model.
One has that

\begin{theorem}[Chulaevsky--Spencer \cite{chsp}]
 Let $\delta > 0$ and $\lambda > 0$ small enough, then
 \be
  L_{\lambda}(E) \geq c_0 \lambda^2
 \ee
 for some $c_0 > 0$ and
 \be\label{eq:Eindeltaset}
  E \in [-2+\delta, -\delta] \cup [\delta, 2 - \delta].
 \ee
\end{theorem}

The approach of Chulaevsky and Spencer was then exploited
by Bourgain and Schlag in \cite{bs} to prove Anderson localization
and H\"older continuity of the integrated density of states.
The restriction \eqref{eq:Eindeltaset} was removed by Avila
and Damanik \cite{ad} and by Sadel and Schulz-Baldes in \cite{ssb}.

All these results have been for small coupling constants
$\lambda$. At my best knowledge, there are currently no
results for large coupling (except \cite{dk}), but there
is work in progress by Avila and Damanik \cite{ad} on it.
In order to state my result, I introduce the following
class of functions

\begin{definition}\label{def:nondegenerate}
 Let $(\Omega,\mu)$ be a probability space, and $f: \Omega\to\R$
 a measurable function. We call $f$ non-degenerate, if there
 are $F, \alpha > 0$ such that for every $E \in \R$ and $\eps > 0$,
 we have that
 \be
  \mu(\{\omega:\quad |f(\omega) - E| \leq \eps\}) \leq F \eps^{\alpha}.
 \ee
\end{definition}

It follows from the \L{}ojasiewicz inequality (see \cite{loj}, Theorem~IV.4.1. in \cite{mal}),
that if $\Omega = \mathbb{T}^K$ and $f: \Omega\to\R$ is real analytic, then
$f$ is non-degenerate in the above sense (see also Lemma~7.3. in \cite{bbook}).
It is necessary that $(\Omega,\mu)$ contains no atoms, such that a function
$f:\Omega\to\R$ can be non-degenerate. We will prove that

\begin{theorem}\label{thm1}
 Let $f: \Omega\to\R$ be non-degenerate.
 There are constants $\lambda_0 = \lambda_0(f) > 0$ and $\kappa = \kappa(f) > 0$
 such that for $\lambda > \lambda_0$ there is a set $\mathcal{E}_b$ of measure
 \be\label{eq:ebadtimes2}
  |\mathcal{E}_b| \leq \E^{- \lambda^{\alpha/2}}
 \ee
 and for $E \notin \mathcal{E}_b$, we have
 \be
  L_{\lambda} (E) \geq \kappa \log(\lambda).
 \ee
\end{theorem}

This question partially answers the question of positivity
of the Lyapunov exponent. Furthermore, if one assumes a
Wegner type estimate for this model, one can show that the
lower bound actually holds for all energies.

One should note here that the claim is non-trivial, since
the set of energies excluded in \eqref{eq:ebadtimes2} is
small compared to the expected size of the spectrum,
which is $|\sigma(H_{\lambda})| \gtrsim \lambda$
as $\lambda\to\infty$, see \cite{delks}.

It should furthermore be remarked, that the proof
of the theorem is done by an iteration of the resolvent
identity combined with an energy elimination mechanisms.
Both of these ideas are not restricted to the one
dimensional setting.
\bigskip

We now turn our attention to the skew-shift. Let $\Omega = \mathbb{T}^2$
and introduce for an irrational number $\alpha$ the map $T_{\alpha}: \Omega\to\Omega$ by
\be\label{eq:def2skew}
 T_{\alpha}(x,y) = (x + \alpha, x +y).
\ee
It turns out that $T_{\alpha}: \Omega\to\Omega$ is uniquely
ergodic and minimal. Understanding the skew-shift is of physical
importance, since it relates to the quantum kicked rotor problem
in the theory of quantum chaos. This problem asks to describe the
behavior of solutions of
\be\label{eq:qkrp}
 i \dot{\psi}(x,t) = a \psi''(x,t) + i b \psi'(x,t) + \kappa \cos(2 \pi x) \left(\sum_{n \in \Z} \delta(t -n)\right) \psi(x,t),
\ee
where $\psi$ is a $1$ periodic function in $x$ and $a,b,\kappa$ are
parameters, see Chapter~16 in \cite{bbook}. Localization for
operators\footnote{One needs to consider more general Toeplitz operators here.}
generated by skew-shift corresponds to quasi--periodic behavior of
solutions of \eqref{eq:qkrp} (see \cite{b2002}).

For a function $f: \Omega\to \R$ and $\lambda >0$,
introduce the potential by
\be
 V_{\lambda,\alpha,\omega}(n) = \lambda f(T_{\alpha}^n\omega).
\ee
Denote by $H_{\lambda,\alpha,\omega} = \Delta + V_{\lambda,\alpha,\omega}$
the associated Schr\"odinger operator and by $L_{\lambda,\alpha}(E)$
its Lyapunov exponent.
In difference, to the doubling map the situation for
large coupling has some understanding. In particular, there
is the following result

\begin{theorem}[Bourgain--Goldstein--Schlag \cite{bgs}, Bourgain \cite{b2002}]
 Let $f:\Omega\to\R$ be a real analytic.
 Given $\eps > 0$, there is $\lambda_0 = \lambda_0(\eps, f) > 0$ such that
 for all $\alpha$ except for a set of measure at most $\eps$,
 we have that
 \be
  L_{\lambda,\alpha}(E) \geq c_0 \log\lambda
 \ee
 for all $E$ and $\lambda > \lambda_0$. Furthermore, Anderson localization
 holds for all $\omega$ not in the exceptional set.
\end{theorem}

We will be able to prove a variant of this, with again eliminating energies,
if we do not assume a Wegner type estimate.

\begin{theorem}\label{thm1b}
 Let $f: \Omega\to\R$ be non-degenerate.
 There are constants $\lambda_0 = \lambda_0(f) > 0$ and $\kappa = \kappa(f) > 0$
 such that for $\lam \geq \lam_0$ there is a set $\mathcal{E}_b = \mathcal{E}_b(\alpha,\lambda)$ of measure
 \be
  |\mathcal{E}_b| \leq \E^{- \lambda^{\alpha/2}}
 \ee
 and for $E \notin \mathcal{E}_b$, we have
 \be
  L_{\lambda,\alpha}(E) \geq \kappa \log(\lambda).
 \ee
\end{theorem}

The current knowledge at small coupling is far from satisfactory. At my
best knowledge the current results are by Bourgain in \cite{b}, \cite{b2},
which prove a statement of the following form.

\begin{theorem}[Bourgain \cite{b}, \cite{b2}]
 Let $f(x,y) = 2 \cos(2\pi y)$, then for each $\lambda > 0$ small enough,
 we may choose $\alpha(\lambda)$ from a set of positive measure such that
 \be
  \lim_{\lambda\to 0} |\{E: L_{\lambda,\alpha(\lambda)}(E) = 0\}| = 0.
 \ee
\end{theorem}

In order to prove this theorem, Bourgain uses approximation of the
skew-shift by rotation, for which he needs $\alpha$ to be small.
In fact, $\alpha(\lambda) \to 0$ as $\lambda\to 0$. Furthermore,
Bourgain does not compute a quantitative lower bound
for the Lyapunov exponent.
\bigskip

Instead of using $\alpha$ as a perturbation parameter, we will
use the dimension $K$ of the torus, on which the skew-shift acts.
For $K \geq 2$ and an irrational $\alpha$, introduce the $K$ skew-shift
$T_{\alpha,K}: \mathbb{T}^K \to \mathbb{T}^K$ by
\be
 (T_{\alpha,K} \omega)_k = \begin{cases} \omega_1 + \alpha & k = 1 \\ \omega_{k} + \omega_{k-1} & k > 1. \end{cases}
\ee
We will furthermore, assume that $f: \mathbb{T} \to \R$ is a $1$ bounded, non-constant
function of mean zero. For $\lambda > 0$, $\alpha$ irrational, $\omega\in\mathbb{T}^K$,
we introduce the potential
\be
 V_{\lambda,\alpha,\omega}(n) = \lambda f((T^{n}_{\alpha,K}\omega)_K).
\ee
We will show that

\begin{theorem}\label{thm2}
 Let $\delta > 0$. There is a a constant $\kappa = \kappa(f) > 0$.
 Let $K \geq 1$ be large enough. There are $\lambda_1 = \lambda_1(\delta,f,K) < \lambda_2 = \lambda_2(\delta,f)$
 with $\lim_{K\to\infty}\lambda_1 = 0$ such that for
 \be
  \lambda_1 \leq \lambda \leq \lambda_2
 \ee
 we have that
 \be
  L_{\lambda,\alpha,K}(E) \geq \kappa \lambda^2
 \ee
 for $E \in [-2+\delta,-\delta] \cup [\delta, 2 -\delta]$
 except in a set of small measure.
\end{theorem}

This result has a major drawback to the one of Bourgain by
not applying to the case of $K = 2$, but needing $K$ large.
However, it has other advantages, like applying to all
irrational $\alpha$, and providing an explicit lower bound.
The restriction of the energy region has similar reasons as the restriction
in the result of Chulaevsky and Spencer, since the method
to verify the initial condition is similar.

It should be pointed out that one can again drop eliminating
energies, if one assumes a Wegner type estimate. This estimate
can be verified explicitly in the case, when
\be
 f(x) = 2 \left(x - \frac{1}{2}\right).
\ee
Then we obtain

\begin{theorem}\label{thm3}
 Let the quantities be as in the previous theorem, and
 $f$ as above, then
 \be
  L_{\lambda,\alpha,K}(E) \geq \kappa \lambda^2
 \ee
 for $E \in [-2+\delta,-\delta] \cup [\delta, 2 -\delta]$
 and $\lam_1 \leq \lam \leq \lam_2$.
\end{theorem}

I believe that using the methods of developed by Bourgain,
Goldstein, and Schlag, one should be able to extend the above
result to all analytic $f$. The main required modification
would be to use results of the form of the matrix-valued
Cartan estimate (see e.g. Chapter~14 in \cite{bbook}) to
prove Wegner type estimates while doing the multiscale
analysis.

%
%
%

\section{Statement of the results}
\label{sec:results}

We will now begin stating the main results of this article.
First, recall that $H$ denotes a discrete one-dimensional
Schr\"odinger operator given by its action on $u \in\ell^2(\Z)$
by
\be
 (Hu)(n) = u(n+1) + u(n-1) + V(n) u(n),
\ee
where $V:\Z\to\R$ is a bounded sequence known as the potential.
We will usually not make the dependance of the operator $H$
and its associated potential explicit.

For $\Lambda\subseteq\Z$ an interval, we denote by $H_{\Lambda}$ the
restriction of $H$ to $\ell^2(\Lambda)$. We furthermore denote
by $e_x$ the standard basis of $\ell^2(\Z)$, that is
$$
 e_x(n) = \begin{cases} 1 & n = x \\ 0 & \text{otherwise}. \end{cases}
$$
For $E \notin \sigma(H_{\Lambda})$ and $x,y\in\Lambda$, we denote by $G_{\Lambda}(E,x,y)$
the Green's function, defined by
\be
 G_{\Lambda}(E,x,y) = \spr{e_x}{(H_{\Lambda} - E)^{-1} e_y}.
\ee
The resolvent equation implies that if $x \in [a,b] \subseteq \Lambda \subseteq \Z$
and $y \in \Lambda\backslash [a,b]$, then
\be
 G_{\Lambda}(E,x,y) = - G_{[a,b]}(E,x,a) G_{\Lambda}(E,a-1,y) - G_{[a,b]}(E,x,b) G_{\Lambda}(E,b+1,y)
\ee
as long as $E \notin \sigma(H_{\Lambda}) \cup \sigma(H_{[a,b]})$. This formula
is a key ingredient in multiscale schemes, since it enables us to go from decay
on small intervals $[a,b]$ to decay on a large interval $\Lambda$. We will
quantify the decay of the Green's function using the following notion.

\begin{definition}\label{def:good}
 For $a \in \Z$ and $K \geq 1$.
 $[a-K,a+K]$ is called $(\gamma, \mathcal{E})$-good if
 \be
  |G_{[a-K, a+K]}(E, a, a \pm K)| \leq \frac{1}{2} \E^{-\gamma K}
 \ee
 for $E \in \mathcal{E}$. Otherwise, $[a-K, a+K]$
 is called $(\gamma, \mathcal{E})$-bad.
\end{definition}

We are now ready to state our first result.

\begin{theorem}\label{thm:mainA}
 Given $0 < \sigma\leq\frac{1}{4}$, $K \geq 1$, $\gamma > 0$, $L \geq 1$,
 and $\mathcal{E} \subseteq \R$ an interval. Assume that
 \be\label{eq:cond1thm}
  \#\{1 \leq l \leq L: [(l-1)K + 1, (l+1)K - 1] \text{ is } (\gamma, K, \mathcal{E})- \text{bad}\} \leq
  \sigma L,
 \ee
 and the following inequalities hold
 \begin{align}
  \label{eq:asA1} \gamma \cdot K & \geq \max\left(\frac{1}{\sigma}, \frac{25}{\sigma} \ln\left(|\mathcal{E}|^{-1}\right)\right) \\
  \label{eq:asA2} \frac{1}{K^3} \E^{\frac{8}{75} \sigma\gamma K} & \geq \frac{2^{17} \E^{3}}{\sigma^4}.
 \end{align}
 Then, there is $\mathcal{E}_0\subseteq\mathcal{E}$ such that
 \be
  |\mathcal{E}_0| \geq (1 - \E^{-\frac{8}{25} \sigma \gamma K}) |\mathcal{E}|
 \ee
 and for $E \in \mathcal{E}_0$, we have that
 \be
  \frac{1}{LK} \log\|\prod_{n=1}^{LK} \begin{pmatrix} V(LK - n) - E & -1 \\ 1  & 0 \end{pmatrix}\| \geq \E^{-8\sigma}\E^{-\frac{1}{99}} \gamma - \frac{\sqrt{2}}{LK}
 \ee
\end{theorem}

The proof of this theorem is our most basic implementation of
multiscale analysis without a Wegner estimate. The main idea is
that for a scale $K_1 \gg K$, one has that the set of
energies, where Wegner type esimates fail has small measure for
each of the sets of the form $l K_1 + [-K_1 +1, K_1 -1]$.
Thus one may achieve that the Wegner estimate holds for
most $l$ outside a set of small measure by using Markov's
inequality. The implementation of this can be
found in Section~\ref{sec:multistep1}. Then Section~\ref{sec:multiscale1}
finishes the proof of this theorem. We furthermore wish
to point out that similar methods of proof have been used
by Bourgain in \cite{b}.

We remark that the requirement $|\mathcal{E}|\geq\E^{-\frac{1}{25} \gamma\sigma K}$
is essential to our approach, since it ensures that the bad set
of energies is smaller then the one, we start with.
Furthermore, it is a non-trivial condition, since by perturbing
the energy in the Green's function will give only a set of
smaller measure. However, one can still use this approach to obtain
this condition, by then decreasing the estimate on the Green's function.
\bigskip

In order to state the next theorem, we will need to phrase
it in the ergodic setting. For this let $(\Omega,\mu)$ be
a probability space and $f: \Omega\to\R$ a bounded real-valued
function. For $\omega\in\Omega$, we introduce a potential
by
\be
 V_{\omega}(n) = f(T^n\omega),\quad n\in\Z.
\ee
We will now write $H_{\omega}$ for $\Delta + V_{\omega}$.
It follows from the ergodic theorem, that \eqref{eq:cond1thm}
is roughly equivalent to
\be\label{eq:cond2thma}
 \mu(\{\omega:\quad [1, 2K-1] \text{ is }(\gamma,\mathcal{E})-\text{bad for }H_{\omega}\}) \leq \sigma.
\ee
In particular, this condition is now independent of $N$. Thus,
one can hope to obtain the conclusion of the previous theorem
for all sufficiently large $N$. In order to exploit this, we
will now introduce the Lyapunov exponent $L(E)$ by
\be
 L(E) = \lim_{N\to\infty} \frac{1}{N}\int_{\Omega} \log \left\|\prod_{n=1}^{N} \begin{pmatrix} V_{\omega}(N - n) - E & - 1 \\ 1 & 0 \end{pmatrix} \right\|d\mu(\omega),
\ee
where the limit exists for all $E$ and defines a subharmonic function.

\begin{theorem}\label{thm:mainB}
 Given $0 < \sigma\leq\frac{1}{4}$, $K \geq 1$, $\gamma > 0$, $L \geq 1$,
 and $\mathcal{E} \subseteq \R$ an interval. Assume
 the inequalities \eqref{eq:asA1}, \eqref{eq:asA2}, and the initial condition
 \eqref{eq:cond2thma}.
 Then, there is $\mathcal{E}_0\subseteq\mathcal{E}$ such that
 \be
  |\mathcal{E}_0| \geq (1 - \E^{-\frac{8}{25} \sigma \gamma K}) |\mathcal{E}|
 \ee
 and for $E \in \mathcal{E}_0$, we have that
 \be
  L(E) \geq \E^{-8\sigma}\E^{-\frac{1}{99}} \gamma.
 \ee
\end{theorem}

This theorem is a combination of the last theorem and properties
of ergodic Schr\"odinger operators. The new parts of the proof
depend on ideas from ergodic theory discussed in Section~\ref{sec:ergodic}
and about the Lyapunov exponent for ergodic Schr\"odinger operators
discussed in Section~\ref{sec:lyap}. The proof is then given in
Section~\ref{sec:proofthmmainB}.

Given this criterion for positive Lyapunov exponent, it seems
a natural question if the assumption \eqref{eq:cond2thma}, can
be checked. It is classical, that \eqref{eq:cond2thma} holds at
large coupling, that is if we consider the family
of potentials
\be
 V_{\omega,\lambda}(n) = \lambda f(T^n\omega)
\ee
for $\lambda > 0$ and $f$ is a nice enough function.
More precisely, we have that.

\begin{proposition}\label{prop:initlarge}
 Assume that $f$ is non-degenerate in the sense of Definition~\ref{def:nondegenerate}.
 Let $E_0 \in \R$, $\sigma > 0$, and introduce
 \begin{align}
  K & = \left\lfloor \frac{\sigma \lambda^{\alpha/2}}{F} \right\rfloor \\
  \gamma & = \frac{1}{5} \log(\lambda) \\
  \mathcal{E} & = [E_0 - 1, E_0 + 1].
 \end{align}
 Assume that $\lambda$ is sufficiently large.
 Then \eqref{eq:cond2thma} holds.
\end{proposition}

For the convenience of the reader, we have included a proof
in Section~\ref{sec:initlarge}. With this proposition, we are ready for
the proof of Theorem~\ref{thm1} and \ref{thm1b}. Instead of proving
them, we will instead proof the following more abstract version.

\begin{theorem}\label{thm:mainC}
 Let $(\Omega,\mu)$ be a probability space, and $f$ a non-degenerate function,
 with $|f(\omega)| \leq 1$. There are constants $\kappa = \kappa(f) > 0$, $\lam_0 = \lam_0(f) > 0$.
 For any $T: \Omega\to\Omega$ ergodic and $\lam > \lam_0$, there exists a set $\mathcal{E}_b = \mathcal{E}_b(T,\lam)$
 of measure
 \be
  |\mathcal{E}_b| \leq \E^{-\lambda^{\alpha/2}},
 \ee
 such that for $E \notin\mathcal{E}_b$
 \be
  L_{T, \lambda} (E) \geq \kappa \log(\lambda).
 \ee
\end{theorem}

\begin{proof}
 It follows from the Combes--Thomas estimate (see Lemma~\ref{lem:ct}) that
 the estimate on the Lyapunov exponent holds for $|E| \geq 3 \lambda$ and $\lambda > 1$.
 Next, observe that we can cover the interval $[-3\lambda,3\lambda]$ with $3\lambda$
 intervals of length $2$ as described in the previous proposition.

 For one of these intervals $\mathcal{E}$, we can the apply Theorem~\ref{thm:mainB}
 with $\sigma=\frac{1}{4}$, $K = \lceil (4F)^{-1} \lambda^{\alpha/2}\rceil$,
 and $\lambda = \frac{1}{5}\log\lambda$. In particular, we see that the arithmetic
 conditions hold for large enough $\lambda$, and the estimate on the size
 of $\mathcal{E}_b$ follows, since the bad set $\mathcal{E} \backslash \mathcal{E}_0$
 has measure behaving like $\E^{-c \log(\lambda) \cdot \lambda^{\alpha/2}}$ for some
 $c > 0$. This finishes the proof.
\end{proof}

The somewhat surprising thing is, that the largeness of
the Lyapunov exponent does not depend on the ergodic
transformation $T$ in this theorem.
\bigskip

Using the Pastur--Figotin \cite{pf} formalism combined with
with large deviation estimates of Bourgain and Schlag \cite{bs},
we are able to obtain an initial condition at small coupling.
In order to state these results, we will need to introduce a bit of notation
about random Schr\"odinger operators. For an integer $N \geq 1$, $\lambda > 0$,
and $\ul{V} \in [-1,1]^N$ introduce the operator $H_{\ul{V},\lam,[0,N-1]}$
acting on $\ell^2([0,N-1])$ by
\be\label{eq:defHulV}
 H_{\ul{V},\lam,[0,N-1]}u(n) = \begin{cases} u(1) + \lam V(0) u(0) & n = 0 \\
 u(n+1) + u(n-1) + \lam V(n) u(n) & 1 \leq n \leq N - 2\\
 u(N-2) + \lam V(N-1) u(N-1) & n = N-1. \end{cases}
\ee
We will show

\begin{proposition}\label{prop:random}
 Let $\nu$ be a probability measure with support in $[-1,1]$
 and mean zero. Introduce
 \be
  \sigma_2 = \int x^2 d\nu,\quad \sigma_4 = \int(x^2 - \sigma_2)^2 d\nu.
 \ee
 Let
 \be
  E = 2 \cos(\kappa) \in (-2,0) \cup (0,2).
 \ee
 Let $A = \min(1, E^2 - 2)$, assume the inequalities
 \begin{align}
  \label{eq:randK}       K &\geq \max\left(\frac{2800 \cdot \sigma_2}{|E| \cdot \sqrt{4 - E^2} \cdot A}, \frac{4608 \sigma_4}{(\sigma_2)^2} \right) \\
  \label{eq:randlam}  \lam &\leq \frac{\sqrt{4 - E^2}}{2} \min\left( \frac{\sigma_2}{7000}, \frac{|E| \cdot  \sqrt{4 - E^2} \cdot A}{4400 \cdot \sigma_2}\right) \\
 \label{eq:randlam2K}  \lambda^2 K & \geq 150000 \frac{ 4 - E^2}{\sigma_2}
%
%
 %
 %
 \end{align}
 Let
 \be
  \gamma = \frac{\lambda^2 \sigma_2}{4 (4 - E^2)}.
 \ee
 Then there exists a set $\mathcal{V}$ satisfying
 \be
  \nu^{\otimes 2 K}(\mathcal{V}) \geq 1 - \frac{1}{16}.
 \ee
 For each $\ul{V} \in \mathcal{V}$, there is $M = M(\ul{V}) \in \{2K-3,2K-2\}$,
 such that the following estimates
 \begin{align}
  |G_{\ul{V}, \lam, [1, M]}(E, 1, K)|&\leq \frac{\sqrt{1 - \frac{|E|}{2}}}{2} \E^{-\gamma K} \\
  |G_{\ul{V}, \lam, [1, M]}(E, M, K)|&\leq \frac{\sqrt{1 - \frac{|E|}{2}}}{2} \E^{-\gamma K} \\
  \|(H_{\ul{V}, \lam, [1, M]} - E)^{-1}\| &\leq \sqrt{1 - \frac{|E|}{2}} 2 K \E^{(\frac{10}{3} \gamma + \log(6)) K}
 \end{align}
 hold.
\end{proposition}

We note the following corollary.

\begin{corollary}
 Under the assumptions of the previous proposition, we have for
 $\ti{E}$ in the set
 \be
  \ti{E} \in \mathcal{E} = [E - \eps, E +\eps],\quad
  \eps = \frac{\E^{-(\ti{\gamma} + (\frac{10}{3} \gamma + \log(6))) K}}{16 \sqrt{1 - \frac{|E|}{2}} K}
 \ee
 for $\ti{\gamma} = \gamma - \frac{1}{K} (\frac{1}{2}\log(1 - \frac{|E|}{2})+\log(2))$ that
 \begin{align}
  |G_{\ul{V},\lam, [1, M]}(\ti{E}, 1, K)|&\leq \frac{1}{2} \E^{-\ti{\gamma} K} \\
  |G_{\ul{V},\lam, [1, M]}(\ti{E}, M, K)|&\leq \frac{1}{2} \E^{-\ti{\gamma} K} \\
 \end{align}
 where $\ul{V} \in \mathcal{V}$ and $M = M(\ul{V}) \in \{2 K - 3, 2K -2\}$.
\end{corollary}

\begin{proof}
 From the resolvent formula, one obtains that
 \begin{align*}
  (H_{\ul{V}, \lam, [1,M]} - \ti{E})^{-1}& = (H_{\ul{V}, \lam, [1,M]} - E)^{-1} \\
  &+ (H_{\ul{V}, \lam, [1,M]} - E)^{-1}
  \cdot \left(\sum_{n=1}^{\infty} \left((\ti{E} - E) (H_{\ul{V}, \lam, [1,M]} - E)^{-1}\right)^n \right).
 \end{align*}
 Hence, we obtain the estimate
 $$
  |G_{\ul{V},\lam, [1, M]}(\ti{E}, 1, K)| \leq \frac{\sqrt{1 - \frac{|E|}{2}}}{2} \E^{-\gamma K}
   + \frac{\eps \|(H_{\ul{V}, \lam, [1, M]} - E)^{-1}\|^2}{1 - \eps \|(H_{\ul{V}, \lam, [1, M]} - E)^{-1}\|}.
 $$
 A quick computation now finishes the proof.
\end{proof}

We will need the following variant of Definition~\ref{def:good}.

\begin{definition}\label{def:good2}
 Given $K \geq 1$, $\mathcal{E} \subseteq \R$ an interval, $\gamma > 0$.
 $\omega\in\Omega$ is called $(K, \mathcal{E}, \gamma)$-good,
 there is $M \in \{2 K -3, 2K-2\}$ such that for $x \in \{K - 1, K\}$
 \be
  |G_{\omega, [1,M]} (E, 1, x)|,
  |G_{\omega, [1,M]} (E, 1, M)| \leq \frac{1}{2} \E^{-\gamma K}
 \ee
 for $E \in \mathcal{E}$.
\end{definition}

We observe that the previous proposition implies, that we are good in
this sense. One can adapt the proof of Theorem~\ref{thm:mainB} in order to
only require to be good in the sense of Definition~\ref{def:good2}
instead of Definition~\ref{def:good}.
\bigskip

We now return to the investigation of ergodic Schr\"odinger operators,
and start by introducing $K$-independence, which will allow us to apply
the tools from random Schr\"odinger operators.

\begin{definition}\label{def:Kindependent}
 Let $(\Omega,\mu)$ be a probability space, $T: \Omega\to\Omega$
 an ergodic transformation, and $f: \Omega\to\R$ bounded and measurable.
 $(\Omega,\mu,T,f )$ is called $K$-independent
 if there exists a probability measure $\nu$ on $\R$ such that
 \be
  \nu^{\otimes K}(\{(f(\omega), f(T\omega), \dots, f(T^{K-1} \omega)):\quad \omega\in A\}) = \mu(A)
 \ee
 for all $A\subseteq\Omega$ measurable.
\end{definition}

In the case of random variables $\Omega = I^{\Z}$, $T$ the left shift,
and $f(\ul{\omega}) = \omega_0$, one clearly has that the system
is $K$ independent for any $K\geq 1$. We furthermore note, the following
lemma which shows how independent the $K$ skew-shift is.

\begin{lemma}
 Let $g:\T \to \R$ be a bounded function, define $f:\T^K\to\R$
 by $f(\ul{\omega}) = f(\omega_K)$. Let $T_{\alpha}$ be the $K$ skew-shift,
 then $(\T^K, \mathrm{Lebesgue}, T_{\alpha},f)$ is $K$ independent.
\end{lemma}

\begin{proof}
 One can check that
 $$
  \begin{pmatrix} (\ul{\omega})_K \\ (T \ul{\omega})_K \\ \vdots \\ (T^{K-1} \ul{\omega})_K \end{pmatrix}
  = \begin{pmatrix} 1 & \ast & \dots & \ast \\ 0 & 1 & \dots &\ast \\ \vdots & & \ddots & \vdots \\ 0 & 0 &\dots & 1 \end{pmatrix} \cdot
  \begin{pmatrix} \omega_1 \\ \omega_2 \\ \vdots \\ \omega_K \end{pmatrix},
 $$
 where $\ast$ denotes a non zero number. This implies the claim
 by the transformation formula for integrals.
\end{proof}

We now come

\begin{theorem}\label{thm:mainD}
 Let $(\Omega,\mu,T,f)$ be $K$-independent.
 Given $\delta >0$, there is a $\kappa = \kappa(f, \delta) > 0$.
 Furthermore, there is $\lambda_1 = \lambda_1(\delta) > 0$ and
 for $K\geq 1$ a $\lambda_2 = \lambda_2(\delta, K) > 0$ with $\lambda_2 \to 0$
 as $K \to \infty$, such that one has for
 \be
  \lambda_2 \leq \lambda \leq \lambda_1,
 \ee
 that
 \be
  L(E) \geq \kappa \lambda^2,
 \ee
 for $E \in [-2+\delta,-\delta] \cup [\delta,2-\delta]$ except in a set,
 whose measure goes to $0$ as $K\to\infty$.
\end{theorem}

\begin{proof}
 In order to ensure that $|\mathcal{E}|$ from the previous corollary
 is large enough, just decrease $\ti{\gamma}$. This finishes
 the proof by an application of Theorem~\ref{thm:mainB}.
\end{proof}

Of course these results have still a major drawback: the need to
eliminate energies. This can be eliminated by assuming Wegner
type estimates, as they are common in the theory of random
Schr\"odinger operators. For this, we will denote by
$\sigma(H_{\omega,\Lambda})$ the spectrum of $H_{\omega,\Lambda}$
given $M \geq 1$, an energy $E$, and $\eps > 0$, we will
be interested in the probability
\be
 \mu(\{\omega:\quad \dist(\sigma(H_{\omega,[0,M-1]}), E) \leq \eps\}),
\ee
which we will need to assume to be small. The most
convenient form of this estimate for us, will be that
\be\label{eq:aswegner1a}
 \mu(\{\omega:\quad \dist(\sigma(H_{\omega,[0,M-1]}), E) \leq \eps\}) \leq C \cdot \frac{M^\beta}{|\log(\eps)|^{\rho}},
\ee
where $C > 0$, $\beta \geq 0$ and $\rho \geq 1$. One has to
restrict here to $0 < \eps \leq \frac{1}{2}$, so one does
not run into problems, when the logarithm becomes $0$.

In the theory of random Schr\"odinger operators, one has as already
mentioned that $V(n)$ are independent identically distributed
random variables. If one assumes, that the density is a bounded
function, one can obtain the following estimate, which is known
as a Wegner estimate
\be
 \mu(\{\omega:\quad \dist(\sigma(H_{\omega,[0,M-1]}), E) \leq \eps\}) \leq C \cdot M \cdot \eps,
\ee
where $C > 0$ is a constant. We will follow the ideas of the proof
and show in Section~\ref{sec:toy} that a similar estimate holds for
the skew-shift model.

Assuming \eqref{eq:aswegner1a}, we are able to remove
the assumption of removing energies from our theorems,
and obtain.

\begin{theorem}\label{thm:mainE}
 Assume the initial length scale \eqref{eq:cond2thma},
 the Wegner type estimate \eqref{eq:aswegner1a},
 \be\label{eq:3beta3rho}
  3 \beta + 3 - \rho \leq 0,
 \ee
 and
 \be\label{eq:condtogetresonant}
  \gamma^\rho K^{\rho - \beta} \sigma^{\rho-1} \geq
  4 \cdot 2^{\beta + \rho} \cdot \E^{(\beta + 1)(4\sigma+\frac{1}{99})} C.
 \ee
 Then
 \be
  L(E) \geq \E^{-\frac{1}{99}} \E^{-4\sigma} \gamma.
 \ee
\end{theorem}

This theorem gives a satisfying criterion for positivity
of Lyapunov exponents, where the conditions exactly correspond
to the ones necessary for localization in the theory
of random Schr\"odinger operators.

Of course \eqref{eq:aswegner1a} is not a simple estimate to
check, since it involves information at all scales. We are
thus only able to check it in the special case of
$f(x) = x - \frac{1}{2}$. This then allows us to prove
the following theorem for the skew--shift at small coupling.

For $\lambda > 0$, we introduce the potential
\be
 V_{\lambda,\alpha, \ul{\omega}}(n) = \lambda f(T_{\alpha}^{n} \ul{\omega}).
\ee
We will show in Section~\ref{sec:toy} the following proposition, which shows
that \eqref{eq:aswegner1a} holds for this family.

\begin{proposition}\label{prop:idsskew}
 Let $H_{\lambda,\alpha,\ul{\omega}} = \Delta + V_{\lambda,\alpha,\ul{\omega}}$.
 Given $\rho \geq 1$, we have for any $E \in \R$ and $M \geq 10$ that
 \be
  \mu(\{\ul{\omega}:\quad \dist(\sigma(H_{\lambda,\alpha,\ul{\omega},[0,M-1]}), E) \leq \eps\})
   \leq 14 \cdot \max(1, \frac{1}{\lambda})  \frac{\rho^{\rho} \cdot M^4}{|\log(\eps)|^{\rho}}.
 \ee
\end{proposition}

Combining this proposition with Proposition~\ref{prop:random} and
Theorem~\ref{thm:mainE}, we can show the following theorem.

\begin{theorem}\label{thm:skew2}
 Given $\eps,\delta > 0$, let
 \be
  E \in [-2+\delta,-\delta] \cup [\delta,2-\delta].
 \ee
 There are constants $C_1 = C_1(\eps, \delta), C_2 = C_2(\delta), \gamma_0 = \gamma_0(\delta) > 0$
 such that for
 \be
   \frac{C_1}{K^{\frac{1}{2} - \eps}} \leq \lam \leq C_2,
 \ee
 and $\alpha$ irrational, we have
 \be
  L_{\lambda,\alpha} (E) \geq \gamma_0 \lambda^2.
 \ee
\end{theorem}

\begin{proof}
 We can assume $0 < \lam < 1$, so by Proposition~\ref{prop:idsskew}, we may
 take
 $$
  \beta = 4,\quad C = \frac{14 \rho^{\rho}}{\lambda}
 $$
 for any $\rho \geq 15$ in \eqref{eq:aswegner1a} ($\rho\geq 15$ such that \eqref{eq:3beta3rho} holds).
 We furthermore, see
 that $d\nu = \chi_{[-1,1]} dx$ and thus
 $$
  \sigma_2 = \frac{2}{3},\quad \sigma_4 = \frac{2}{5}.
 $$
 We can thus choose $\gamma = \frac{\lam^2}{48 \sin(\kappa)}$.
 We may choose $\sigma =\frac{1}{4}$, and thus \eqref{eq:condtogetresonant}
 becomes
 $$
  \frac{\lam^{2 \rho + 1}}{\sin(\kappa)^\rho} K^{\rho - 3} \geq C_1 (384 \rho) ^{\rho}
 $$
 for some constant $C_1 > 0$. This finishes the proof by applying Theorem~\ref{thm:mainE}
 to the initial condition obtained by Proposition~\ref{prop:random}.
\end{proof}

Proving positive Lyapunov exponent is not the only problem
concerning ergodic Schr\"odinger operators. There is probably
an even larger literature as how to go from positive Lyapunov
exponent to Anderson localization (see for example \cite{j}
and \cite{bg} in the case of rotations). However, one cannot
expect Anderson localization to hold in the generality
discussed in this paper, since for example the results of
Avron and Simon in \cite{avsi} show, that if $T: \Omega\to\Omega$
is well approximated by periodic transformation, then the
spectrum of $H$ is purely continuous, and hence Anderson
localization cannot hold.

We now give an overview, of what happens in the following sections.
Section~\ref{sec:ergodic} derives some consequences of the
ergodic theorem, which will be needed in the following.
Section~\ref{sec:lyap} discusses properties of the Lyapunov
exponent, which will be needed. Section~\ref{sec:multistep1},
\ref{sec:multiscale1}, and \ref{sec:proofthmmainA} contain the
proof of Theorem~\ref{thm:mainA}. Then Theorem~\ref{thm:mainB}
is proven in Section~\ref{sec:proofthmmainB}. Proposition~\ref{prop:initlarge}
is proven in Section~\ref{sec:initlarge} and Proposition~\ref{prop:random}
in Sections~\ref{sec:pf} and \ref{sec:ldtrandom}. Sections~\ref{sec:multistep2}
and \ref{sec:multiscale2} contain the proof of Theorem~\ref{thm:mainE}.
Finally Proposition~\ref{prop:idsskew} is proven in Section~\ref{sec:toy}.

%
%
%

\section{Ergodic Theory}
\label{sec:ergodic}

In this section, we review the notions of ergodic theory, we will use. As usual,
we denote by $(\Omega,\mu)$ a probability space and by $T:\Omega\to\Omega$
an ergodic transformation, that is if $A\subseteq\Omega$ satisfies
$T^{-1} A = A$ almost everywhere, then $\mu(A) \in \{0,1\}$. We recall that
the mean ergodic theorem tells us, that if $f$ is a function in $L^2(\Omega,\mu)$,
then its averages
\be
 f_N(\omega) = \frac{1}{N} \sum_{n=0}^{N-1} f(T^n \omega)
\ee
converge to $\int_{\Omega} f(\omega) d\mu(\omega)$ in $L^2(\Omega,\mu)$.
This result will be the mean ingredient of ergodic theory, we will use.
However, some of the results from ergodic Schr\"odinger operators, we are
using depend on the somewhat different Birkhoff ergodic theorem, saying that
one has pointwise convergence almost everywhere.

We will be interested in the following question: Given the good set
$\Omega_g \subseteq \Omega$ and an integer $K \geq 1$,
can we choose a large set of $\omega$ such that, we have
$$
 T^{l K} \omega \in \Omega_g
$$
for a set of $l$ with density close to $\mu(\Omega_g)$. The following
lemma does exactly this.

\begin{lemma}\label{lem:ergodic1}
 Given $\Omega_g \subseteq \Omega$, $0 < \kappa < 1$, $K \geq 1$.
 Then, there exists $\Omega_0 \subseteq \Omega$ such that
 for $\omega \in \Omega_0$, there is a sequence $L_t = L_t(\omega) \to \infty$
 such that
 \be
  \frac{1}{L_t} \#\{0 \leq l \leq L_t - 1:\quad T^{l K} \omega \in \Omega_g\} \geq \kappa \mu(\Omega_g)
 \ee
 and $\mu(\Omega_0) > 0$.
\end{lemma}

\begin{proof}
 Letting $f = \chi_{\Omega_0}$ in the mean ergodic theorem, we find that
 $$
  \lim_{N\to\infty} \int_{\Omega} \left|\frac{1}{N} \#\{0 \leq n \leq N-1:\quad T^n \omega \in \Omega_g\} - \mu(\Omega_g)\right|^2 d\mu(\omega) = 0.
 $$
 Thus, we obtain in particular
 $$
  \lim_{N\to\infty} \mu(\{\omega:\quad \frac{1}{N} \#\{0 \leq n \leq N-1:\quad T^n \omega \in \Omega_g\} < \kappa \mu(\Omega_g)\}) = 0.
 $$
 We thus may find a set $\Omega_1$ of positive measure, such that for each $\omega \in \Omega_1$, there
 is a sequence $N_t = N_t(\omega)$ going to $\infty$ such that
 $$
  \frac{1}{N_t} \#\{0 \leq n \leq N_t-1:\quad T^n \omega \in \Omega_g\} \geq \kappa \mu(\Omega_g).
 $$
 For each $\omega\in\Omega_1$, we may find an $0 \leq s = s(\omega) \leq K-1$
 such that $N_t \pmod K = s$ for infinitely many $t$. Introduce
 $$
  \Omega_0 =\{T^{-s(\omega)}\omega:\quad \omega\in\Omega_1\},
 $$
 and choose for $\omega\in\Omega_0$ the sequence $L_t = \frac{N_t}{K}$, for
 the $N_t$ with $N_t \pmod K = s$. The claim now follows by construction.
\end{proof}

Furthermore recall that a transformation $T: \Omega\to\Omega$ is called
totally ergodic, if for every $n\geq 1$ the transformation $T^n: \Omega\to\Omega$
is ergodic. Total ergodicity allows us to not need the step of passing
to a subsequence in the proof of the last lemma. Thus, we may conclude
that

\begin{lemma}\label{lem:ergodic2}
 Suppose that $T: \Omega\to\Omega$ is totally ergodic.
 Given $\Omega_g \subseteq \Omega$, $0 < \kappa < 1$, $0 < \tau < 1$, and $K \geq 1$.
 There is $\Omega_0 \subseteq \Omega$ such that
 for $\omega\in\Omega_0$ and $L$ large enough
 \be
  \frac{1}{L}\#\{0 \leq l \leq L - 1:\quad T^{l K} \omega\in\Omega_g\}\geq \kappa\mu(\Omega_g)
 \ee
 and
 \be
  \mu(\Omega_0) \geq \tau.
 \ee
\end{lemma}

%
%
%

\section{The Lyapunov exponent}
\label{sec:lyap}

We let again $(\Omega,\mu)$ be a probability space,
$f: \Omega\to\R$ a bounded measurable function,
$T:\Omega\to\Omega$ an invertible ergodic transformation,
and set $V_{\omega}(n) = f(T^n\omega)$ for $\omega\in\Omega$ and $n\in\mathbb{Z}$.
Introduce the $N$ step transfer matrix $A_{\omega}(E, N)$ by
\be
 A_{\omega}(E, N) = \prod_{n=1}^{N} \begin{pmatrix} E - V_{\omega}(N - n) & - 1 \\ 1 & 0 \end{pmatrix}.
\ee
Let $u$ be a solution of $H_{\omega} u = E u$ interpreted as a difference equation.
Then we have that
\be
 \begin{pmatrix} u(N+1) \\ u(N) \end{pmatrix} =
 A_{\omega}(E, N) \cdot \begin{pmatrix} u(1) \\ u(0) \end{pmatrix},
\ee
explaining the name. Define the Lyapunov exponent by
\be\label{eq:deflyap2}
 L(E) = \lim_{N\to\infty} \frac{1}{N}\int_{\omega} \log \left\|\prod_{n=1}^{N} \begin{pmatrix} V_{\omega}(N - n) - E & - 1 \\ 1 & 0 \end{pmatrix} \right\|d\mu(\omega),
\ee
where the limit exists because of submultiplicativity of the matrix norm,
which implies that the sequence
$$
 \frac{1}{N}\int_{\omega} \log \left\|\prod_{n=1}^{N} \begin{pmatrix} V_{\omega}(N - n) - E & - 1 \\ 1 & 0 \end{pmatrix} \right\|d\mu(\omega)
$$
is subadditive.
Furthermore, the following lemma was shown by Craig and Simon in \cite{cs2}.

\begin{lemma}\label{thm:subharmonic}
 The function $L(E)$ is subharmonic in $E$.
\end{lemma}

We will mainly use the upper semicontinuity provided by this result.
The next result will allow us to go from Green's function estimates
to estimates for the Lyapunov exponent.

\begin{lemma}\label{lem:greentolyap}
 If
 \be
  |G_{\omega, \Lambda}(E,k, N)| \leq \E^{-\gamma N}
 \ee
 for $\Lambda \in \{[0,N], [1,N]\}$, $k \in \{k_0-1, k_0\}$, then
 \be
  \frac{1}{N} \log\|A_{\omega}(E,N)\| \geq \gamma - \frac{\log\sqrt{2}}{N}.
 \ee
\end{lemma}

\begin{proof}
 We first observe, that
 $$
  A_{\omega}(E, n) = \begin{pmatrix} c_{\omega,E}(n) & s_{\omega,E}(n) \\
   c_{\omega,E}(n-1) & s_{\omega,E}(n-1) \end{pmatrix},
 $$
 where these solve
 $$
  H_{\omega} c_{\omega,E} = E c_{\omega,E},\quad H_{\omega} s_{\omega,E} = E s_{\omega,E},
 $$
 with initial conditions
 $$
  \begin{pmatrix} c_{\omega,E}(0) & s_{\omega,E}(0) \\ c_{\omega,E}(-1) & s_{\omega,E}(-1) \end{pmatrix}
  = \begin{pmatrix} 1 & 0 \\ 0 & 1 \end{pmatrix}.
 $$
 We let $u_{\omega,E}$ be the solution of $H_{\omega} u_{\omega,E} = E u_{\omega,E}$,
 that satisfies $u_{\omega,E}(N) = 1$ and $u_{\omega,E}(N+1) = 0$. We then find
 for $x \leq y$, that
 \begin{align*}
  G_{\omega, [0,N]}(E, x, y) &= \frac{c_{\omega,E}(x) u_{\omega,E}(y)}{W(c_{\omega,E}, u_{\omega,E})},&
  G_{\omega, [1,N]}(E, x, y) &= \frac{s_{\omega,E}(x) u_{\omega,E}(y)}{W(s_{\omega,E}, u_{\omega,E})},
 \end{align*}
 where $W(u, v) = u(n+1) v(n) - u(n) v(n+1)$ is the Wronskian. One can check that
 $W(u,v)$ is independent of $n$ if $u$ and $v$ solve $H_{\omega} u = E u$, $H_{\omega} v = E v$.
 Evaluating the Wronskian at $N$ yields
 $$
  W(c_{\omega,E}, u_{\omega,E}) = - c_{\omega,E}(N+1),\quad W(s_{\omega,E}, u_{\omega,E}) = - s_{\omega,E}(N+1).
 $$
 Hence, we obtain the formulas
 $$
  |G_{\omega, [0,N]}(E, x, N)| = \left|\frac{c_{\omega,E}(x)}{c_{\omega,E}(N+1)}\right|,\quad
  |G_{\omega, [1,N]}(E, x, N)| = \left|\frac{s_{\omega,E}(x)}{s_{\omega,E}(N+1)}\right|,
 $$
 Since $\det(A_{\omega}(E, k_0) = 1$, it follows that
 $$
  \min(|c_{\omega,E}(k_0)|,|c_{\omega,E}(k_0-1)|,|s_{\omega,E}(k_0)|,|s_{\omega,E}(k_0-1)|) \geq \frac{1}{\sqrt{2}}.
 $$
 Hence, we see that
 \begin{align*}
  \min_{k \in {k_0-1, k_0},\, a \in \{0,1\}} |G_{\omega, [a,N]}(E, k, N)| &\geq \frac{1}{\sqrt{2}} \cdot
  \min\left(\frac{1}{|c_{\omega,E}(N+1)|}, \frac{1}{|s_{\omega,E}(N+1)|}\right)\\
  & \geq \frac{1}{\sqrt{2}} \cdot \|A_{\omega}(E,N)\|
 \end{align*}
 taking logarithms and dividing by $N$ implies the result.
\end{proof}

This lemma will allow us to go from estimates on the
Green's function to estimates on the Lyapunov exponent.
One should furthermore observe, that in order to conclude
in the general setting, that $L(E) > 0$, one would
need information for all large $N$.
However, in the ergodic setting one can relax this
a little bit. By a Theorem of Craig and Simon, we have that

\begin{theorem}\label{thm:cs}
 Introduce
 \be
  \ol{L}(E,\omega) = \limsup_{n \to \infty} \frac{1}{n} \log\|A_{\omega}(E,n)\|.
 \ee
 Then there exists $\Omega_{CS} \subseteq \Omega$
 of measure $\mu(\Omega_{CS}) = 1$, such that
 \be
  \ol{L}(E,\omega) \leq L(E)
 \ee
 for $\omega \in \Omega_0$.
\end{theorem}

\begin{proof}
 This is Theorem~2.3 in \cite{cs2}.
\end{proof}

We will call $\Omega_{CS}$ the Craig--Simon set.
We note the following consequence

\begin{lemma}\label{lem:poslyaptest}
 Suppose, we are given $\gamma > 0$, $\omega \in \Omega_{CS}$
 and for $k \geq 1$ integers $n_k \to \infty$ such that
 \be\label{eq:estigreennk}
  |G_{\Lambda,\omega}(E,x,y)| \leq \E^{-\gamma n_k}
 \ee
 for $\Lambda \in \{[0,n_k], [1,n_k]\}$, $x \in \{x_0, x_0 + 1\}$,
 some $x_0$ and $y \in \partial\Lambda$. Then
 \be
  L(E) \geq \gamma.
 \ee
\end{lemma}

\begin{proof}
 By Lemma~\ref{lem:greentolyap}, we have that \eqref{eq:estigreennk}
 implies that $\ol{L}(E) \geq \gamma$. Now, the claim follows
 from Theorem~\ref{thm:cs}.
\end{proof}

%
%
%

\section{The multiscale step}
\label{sec:multistep1}

In this section, we will begin with the exposition of our adaptation of multiscale
analysis. For this, we will not work with an ergodic potential, but will assume
that $\{V(n)\}_{n=0}^{N-1}$ is any real valued sequence of $N$ numbers. We then
define $H$ as the corresponding Schr\"odinger operators on $\ell^2([0,N-1])$
and denote by $H_{\Lambda}$ the restrictions to intervals $\Lambda\subseteq [0,N-1]$.
This generality is mainly used to simplify the notation, and to make clear,
when ergodicity enters.

Furthermore, since we do not make quantitative assumptions on the recurrence
properties of $T: \Omega\to\Omega$, it is necessary to work in this section
with intervals of varying length. However, this does not create major technical
difficulties, since their boundary still consists of only two points.

We now start by defining our basic notion of a good sequence $\{V(n)\}_{n=0}^{N-1}$.

\begin{definition}\label{def:critical}
 Let $0 < \delta < 1$, $0 < \sigma \leq\frac{1}{4}$, $\mathcal{E} \subseteq \mathbb{R}$ an interval, and $L \geq 1$.

 A sequence $\{V(n)\}_{n=0}^{N-1}$ is called $(\delta,\sigma, L,\mathcal{E})$-critical,
 if there are integers
 \be
  0 \leq k_0 < k_1 < k_2 < k_3 < \dots < k_L < k_{L+1} \leq N,
 \ee
 and a set $\mathcal{L} \subseteq [1, L]$ such that
 \be\label{eq:condcalL}
  \frac{\#\mathcal{L}}{L} \leq \sigma.
 \ee
 And for $l \notin \mathcal{L}$, we have that
 \be
  |G_{[k_{l-1} + 1, k_{l+1} - 1]}(E, k_{l}, k_{l \pm 1}\mp 1)| \leq \frac{1}{2} \E^{-\delta}
 \ee
 for $E \in \mathcal{E}$.
\end{definition}

In order to state the next theorem, we have to explain a division of
$\mathcal{E} = [E_0,E_1]$ into $Q$ intervals of length $\approx\E^{-\sigma\delta}$.
Introduce $Q = \lceil (E_1 - E_0) \E^{\sigma\delta} \rceil$, and
\be
 \mathcal{E}_q  = \left[E_0 + q \frac{E_1 - E_0}{Q}, E_0 + (q+1) \frac{E_1 - E_0}{Q}\right],
\ee
for $q = 0, \dots, Q -1$. If
\be\label{eq:asE1E0}
 E_1 - E_0 \geq \E^{-\sigma\delta}
\ee
holds, we have that
\be\label{eq:ineqsizeQ}
 (E_1 - E_0) \E^{\sigma\delta} \leq Q \leq 2 (E_1 - E_0) \E^{\sigma\delta}
\ee
and for all $q$
\be\label{eq:estisizeEq}
 \frac{1}{2} \E^{-\sigma\delta} \leq |\mathcal{E}_q| \leq \E^{-\sigma\delta}.
\ee
The main result of this section will be

\begin{theorem}\label{thm:multistep1}
 Assume that $\{V(n)\}_{n=0}^{N-1}$ is $(\delta,\sigma,L, \mathcal{E})$-critical,
 $M \geq 3$,
 \be\label{eq:condsigmaLM}
  \frac{\sigma L}{M} \geq 2,
 \ee
 and $\sigma\leq\frac{1}{4}$. Introduce
 \be\label{eq:deftisigma}
  \ti{\sigma} = \frac{1}{2}\sigma
 \ee
 and
 \be\label{eq:deftidelta}
  \ti{\delta} = (1 - 2 \sigma) M \delta.
 \ee
 Then there exists a set $\mathcal{Q} \subseteq [0, Q-1]$ and $\ti{L} \geq 1$ such that
 \be\label{eq:boundsizecalQ}
  \#\mathcal{Q} \leq \frac{2^{15}}{\ti{\sigma}} \left(\frac{(M+1)}{\sigma} \cdot \frac{N}{L}\right)^3
 \ee
 and
 \be\label{eq:sizetiL}
  (1 - 2 \sigma) \frac{L}{M+1} \leq \ti{L} \leq \frac{L}{M+1}
 \ee
 and for $q \notin\mathcal{Q}$, we have that $\{V(n)\}_{n=0}^{N-1}$
 is also $(\ti{\delta}, \ti{\sigma}, \ti{L}, \mathcal{E}_q)$-critical.
\end{theorem}

We observe that in our case $N \gtrsim L$, so \eqref{eq:condsigmaLM} will be
satisfied for all large enough $N$. The rest of this section is spent proving
the above theorem.

We will now describe how we choose the sequence $\ti{k}_l$ given the integer
$M \geq 1$ from Theorem~\ref{thm:multistep1}. This will be the sequence, we check
Definition~\ref{def:critical} with. First pick
\be\label{eq:deftik1}
 \ti{k}_0 = k_0.
\ee
Now assume that we are given $\ti{k}_s = k_{l_s}$ for $0 \leq s \leq j$,
then we choose $\ti{k}_{j+1} = k_{l_{j+1}}$ so that
\be\label{eq:deftik2}
 \#\{l \notin \mathcal{L}:\quad \ti{k}_j < k_l < \ti{k}_{j+1}\} = M.
\ee
This procedure stops once, we would have to choose $\ti{k}_{j+1} > N$.
We will call the maximal $l$ so that $\ti{k}_{l+1}$ is defined
$\ti{L}$. This means that we have now defined
$$
 0 \leq \ti{k}_0 < \ti{k}_1 < \dots < \ti{k}_{\ti{L}} < \ti{k}_{\ti{L} + 1} \leq N - 1.
$$
We have the following lemma

\begin{lemma}
 Assume $\sigma\frac{L}{M} \geq 2$, that is \eqref{eq:condsigmaLM}, then we have that
 \be\label{eq:lowboundtiL}
  \ti{L} \geq (1 - 2 \sigma) \frac{L}{M+1}.
 \ee
\end{lemma}

\begin{proof}
 By \eqref{eq:condcalL}, we have that
 $$
  \#([1,L] \backslash \mathcal{L}) \geq (1 - \sigma) L.
 $$
 We observe now, that $l_{j+1} - l_{j} \geq M + 1$, and even
 $$
  l_{j+1} - l_j = M + 1 + \#\{l\in\mathcal{L}: \ti{k}_j < k_l < \ti{k}_{j+1}\}.
 $$
 Hence, we may choose
 $$
  \ti{L} \geq (1 - \sigma) \frac{L}{M+1} - 2
 $$
 the claim now follows by $2\leq\sigma\frac{L}{M+1}$.
\end{proof}

We furthermore have the following estimate

\begin{lemma}
 Assume $\sigma\leq\frac{1}{4}$. Let
 \be\label{eq:defticalL0}
  \widetilde{\mathcal{L}}_0 = \left\{l:\quad \ti{k}_{l+1} - \ti{k}_{l-1} \geq \frac{16 N (M+1)}{\sigma L}\right\}
 \ee
 then, we have that
 \be\label{eq:sizetiL0}
  \frac{\#\widetilde{\mathcal{L}}_0}{\ti{L}} \leq \frac{1}{2} \ti{\sigma}.
 \ee
\end{lemma}

\begin{proof}
 Since $0\leq \ti{k}_0 \leq \ti{k}_{\ti{L}+1}\leq N$, we have that
 $$
  \sum_{l=1}^{\ti{L}} (\ti{k}_{l+1} - \ti{k}_{l-1}) = \ti{k}_{\ti{L}+1} - \ti{k}_0 + \ti{k}_{\ti{L}} - \ti{k}_1 \leq 2 N.
 $$
 Now, Markov's inequality implies that
 $$
  \#\widetilde{\mathcal{L}}_0 \leq \left(\frac{1}{2} \cdot \frac{\sigma}{2}\right)  \cdot \left(\frac{L}{2 (M+1)}\right).
 $$
 By \eqref{eq:lowboundtiL} and $\sigma\leq\frac{1}{4}$, we have that
 $\frac{1}{\ti{L}}\leq \frac{2 (M+1)}{L}$. Now, the claim follows from
 $\ti{\sigma} = \frac{\sigma}{2}$ and the above equation.
\end{proof}

Before coming to the next lemma, we will first introduce the
notion of non-resonance.

\begin{definition}\label{def:nonresonant}
 Given an interval $I \subseteq [0,N-1]$, an energy interval $\mathcal{E}$,
 and $\eps > 0$.  $\{V(n)\}_{n=0}^{N-1}$ is called $(I, \mathcal{E}, \eps)$ non-resonant,
 if for every $\Lambda\subseteq I$, we have that
 \be\label{eq:condELambda}
  \dist(E, \sigma(H_{\Lambda})) \geq \eps
 \ee
 for all $E \in \mathcal{E}$. Otherwise,
 $\{V(n)\}_{n=0}^{N-1}$ is called $(I, \mathcal{E}, \eps)$ resonant.
\end{definition}

Introduce the set $\mathfrak{L}_q$ for $0 \leq q \leq Q$ by
\be
 \mathfrak{L}_q = \{1 \leq l \leq \ti{L}:\quad \{V(n)\}_{n=0}^{N-1}
  \text{ is }([\ti{k}_{l-1}, \ti{k}_{l+1}], \mathcal{E}_q, 2 \E^{-\sigma\delta})\text{ resonant}\}.
\ee
We will now discuss the size of this set.

\begin{lemma}\label{lem:energyelem}
 There is a set $\mathcal{Q}$ such that
 \be
  \#\mathcal{Q} \leq \frac{2^{15}}{\ti{\sigma}} \left(\frac{N (M+1)}{\sigma L}\right)^3
 \ee
 and for $q\notin\mathcal{Q}$, we have that
 \be\label{eq:sizefrakLq}
  \frac{\#\mathfrak{L}_q}{\ti{L}} \leq \ti{\sigma}.
 \ee
\end{lemma}

The estimate on $\#\mathcal{Q}$ is not sharp. By a more careful analysis,
the power in $\left(\frac{N (M+1)}{\sigma L}\right)^3$ could be lowered
to $\left(\frac{N (M+1)}{\sigma L}\right)$. However, we have decided not
to pursue this, since the overall improvement is minor. In order to
achieve this, one has to make explicit in Lemma~\ref{lem:greenimp} for
which intervals the non-resonance condition is being used, and only
assume it for them.

\begin{proof}[Proof of Lemma~\ref{lem:energyelem}]
 For $l$ introduce
 $$
  g(l) = \#\{q:\quad \{V(n)\}_{n=0}^{N-1}\text{ is }([\ti{k}_{l-1}, \ti{k}_{l+1}], \mathcal{E}_q, 2 \E^{-\sigma\delta}) \text{ resonant}\}.
 $$
 We will now derive an upper bound on $g(l)$. First note that $\sigma(H_{\Lambda})$
 consists of $\#\Lambda$ elements, so
 $$
  \bigcup_{\Lambda \subseteq [\ti{k}_{l-1}, \ti{k}_{l+1}]} \sigma(H_{\Lambda})
 $$
 consists of at most $(\ti{k}_{l+1} - \ti{k}_{l-1})^3$ elements. For each
 $E$ in the above set, we have that its $2 \E^{-\sigma\delta}$ neighborhood
 can intersect at most $8$ of the $\mathcal{E}_q$ intervals. Thus, we have that
 $$
  g(l) \leq 8 (\ti{k}_{l+1} - \ti{k}_{l-1})^3.
 $$
 In particular for $l\notin\widetilde{\mathcal{L}}_0$, we have
 by \eqref{eq:defticalL0} that
 $$
  g(l) \leq 2^{15} \left(\frac{N (M+1)}{\sigma L}\right)^3.
 $$
 Let $h(q) = \#\mathfrak{L}_q$, so that
 $$
  h(q) \leq \#\{l \notin\widetilde{\mathcal{L}}_0:\quad
  \{V(n)\}_{n=0}^{N-1}\text{ is }([\ti{k}_{l-1}, \ti{k}_{l+1}], \mathcal{E}_q, 2 \E^{-\sigma\delta}) \text{ resonant}\}.
 $$
 We obtain
 $$
  \sum_{q=0}^{Q-1} h(q) \leq \sum_{l\notin\widetilde{\mathcal{L}}_0} g(l) \leq 2^{15} \ti{L} \left(\frac{N (M+1)}{\sigma L}\right)^3.
 $$
 Let $\mathcal{Q}$ be the set
 $$
  \mathcal{Q} = \{q:\quad h(q) \geq \ti{\sigma} \ti{L}\},
 $$
 now the claim follows from Markov's inequality.
\end{proof}

We observe that \eqref{eq:condELambda} implies that
\be
 \|(H_{\Lambda} - E)^{-1}\| \leq \frac{1}{2} \E^{\sigma\delta}.
\ee

\begin{lemma}\label{lem:greenimp}
 Assume for $(l,q)$ that $\{V(n)\}_{n=0}^{N-1}$ is $([\ti{k}_{l-1}, \ti{k}_{l+1}], \mathcal{E}_q, 2 \E^{-\sigma\delta})$
 non-resonant, then
 \be
  |G_{[\ti{k}_{l-1}+1, \ti{k}_{l+1}-1]}(E, \ti{k}_l, \ti{k}_{l\pm1}\mp 1)| \leq \frac{1}{2} \E^{-\ti{\delta}}
 \ee
 for $E \in \mathcal{E}_q$.
\end{lemma}

\begin{proof}
 Let $x = \ti{k}_{l\pm1}$ (one of the two). Since \eqref{eq:condELambda}, we have that
 $$
  |G_{[\ti{k}_{l-1}+1, \ti{k}_{l+1}-1]}(E, \ti{k}_l, x)| \leq \frac{1}{2} \E^{-\sigma\delta}.
 $$
 By construction of $\ti{k}_l$, we have sets $\mathcal{J}_{\pm}$ such that for $j \in \mathcal{J}_{\pm}$
 we have $[k_{j-1}, k_{j+1}] \subseteq [\ti{k}_{l}, \ti{k}_{l\pm 1}] \cup [\ti{k}_{l\pm 1}, \ti{k}_{l}]$.
 Furthermore, for $j \in \mathcal{J}_{+} \cup \mathcal{J}_-$, we have that
 $$
  |G_{[k_{j-1}+1, k_{j+1}-1]}(E, k_j, k_{j\pm1}\mp1)| \leq \frac{1}{2}\E^{-\delta}
 $$
 for $E \in \mathcal{E}_q\subseteq\mathcal{E}$.

 By the resolvent equation, we find that
 \begin{align*}
  |G_{[\ti{k}_{l-1}+1, \ti{k}_{l+1}-1]}(E, \ti{k}_l, x)|
   \leq \frac{1}{2} \E^{-\sigma\delta} \bigg( &|G_{[\ti{k}_{l-1}+1, \ti{k}_{l+1}-1]}(E, k_{j_-}, x)| \\
   &+ |G_{[\ti{k}_{l-1}+1, \ti{k}_{l+1}-1]}(E, k_{j_+}, x)| \bigg),
 \end{align*}
 where $j_+ = \max(\mathcal{J}_+)$ and $j_- = \min(\mathcal{J}_-)$. Now,
 by the decay of the Green's function, we know that
 \begin{align*}
  |G_{[\ti{k}_{l-1}+1, \ti{k}_{l+1}-1]}(E, \ti{k}_l, x)|
   \leq \frac{1}{4} \E^{-(1 - \sigma)\delta} \bigg( &|G_{[\ti{k}_{l-1}+1, \ti{k}_{l+1}-1]}(E, k_{j_--1}+1, x)| \\
   &+ |G_{[\ti{k}_{l-1}+1, \ti{k}_{l+1}-1]}(E, k_{j_-+1}-1, x)| \\
   &+ |G_{[\ti{k}_{l-1}+1, \ti{k}_{l+1}-1]}(E, k_{j_+-1}+1, x)| \\
   &+ |G_{[\ti{k}_{l-1}+1, \ti{k}_{l+1}-1]}(E, k_{j_++1}-1, x)| \bigg).
 \end{align*}
 We may iterate this procedure $M = \#\mathcal{J}_+ = \#\mathcal{J}_-$ many
 times, proving the proposition by our choice of $\ti{\delta}$.
\end{proof}

\begin{proof}[Proof of Theorem~\ref{thm:multistep1}]
 We are essentially done. We observe, that for $q \notin\mathcal{Q}$, we can choose
 $\mathfrak{L} =  \mathfrak{L}_q$, which satisfies
 $$
  \frac{\#\mathfrak{L}}{\ti{L}} \leq \ti{\sigma},
 $$
 by \eqref{eq:sizefrakLq}. Furthermore, we then
 have the estimate on the Green's function on $[\ti{k}_{l-1}, \ti{k}_{l+1}]$
 by the last lemma for $l \notin\mathfrak{L}$. This finishes the
 proof that $\{V(n)\}_{n=0}^{N-1}$ is $(\ti{\delta}, \ti{\sigma}, \ti{L}, \mathcal{E}_q)$-critical.
\end{proof}

%
%
%

\section{Inductive use of the multiscale step}
\label{sec:multiscale1}

In this section, we develop an inductive way to apply Theorem~\ref{thm:multistep1}.
This will lead in the following section to the proof of Theorem~\ref{thm:mainA}.
A major part of this section is taken up by checking inequalities between various
numerical quantities, necessary to show that everything converges.

Given numbers $\delta > 0$ and $0 < \sigma \leq \frac{1}{4}$, we will first introduce
$\delta_j$, $\sigma_j$, and $M_j$.
Introduce $\delta_0 = \delta$ and
\begin{align}
 M_j          &= 100^{j+1} \\
 \sigma_j     &= \frac{1}{2^j} \sigma\\
 \delta_{j+1} &= (1 - 2\sigma_j) M_j \delta_j.
\end{align}
This choice is motivated by \eqref{eq:deftisigma} and \eqref{eq:deftidelta}.
We first observe that

\begin{lemma}
 We have that
 \begin{align}
  \label{eq:productMj}             \prod_{k=0}^{j} M_k   &= 10^{(j+1)(j+2)} = 10^{j^2} \cdot 1000^j \cdot 100\\
  \label{eq:lowbounddeltaj}        \delta_j &\geq \E^{-4\sigma} 10^{(j+1)(j+2)} \delta \\
  \label{eq:lowbounddeltajsigmaj}  \sigma_j \delta_j & \geq \E^{-4\sigma} 10^{j^2} 500^{j} 100 \sigma\delta.
 \end{align}
\end{lemma}

\begin{proof}
 For \eqref{eq:productMj}, observe that
 $$
  \prod_{k=0}^{j} M_k = 100^{\sum_{k=0}^{j} (k+1)}
 $$
 and $\sum_{k=0}^{j} (k+1) = \frac{(j+1)(j+2)}{2}$.

 For \eqref{eq:lowbounddeltaj}, we have that $\delta_{j + 1} = \prod_{k=1}^{j} (1 - \frac{\sigma}{2^k}) M_k \cdot \delta$,
 and since $\prod_{k=1}^{j}(1 - 2 \frac{\sigma}{2^k}) \geq \prod_{k=1}^{\infty}(1 - 2 \frac{\sigma}{2^k})$,
 we have that
 $$
  \prod_{k=1}^{j}(1 - 2\frac{\sigma}{2^k}) \geq \exp\left(\sum_{k=1}^{\infty} \log(1 - 2\frac{\sigma}{2^j})\right).
 $$
 Now using that $\log(1 - x) \geq - 2x$ for $0 < x < 1/2$, we have that
 $\sum_{j=1}^{\infty} \log(1 - 2 \frac{\sigma}{2^j}) \geq -4 \sigma \sum_{k=1}^{\infty} \frac{1}{2^k} = - 4\sigma$
 and thus the inequalities follow.
\end{proof}

We let $L_j$ be a sequence of numbers, that satisfies
\be\label{eq:recurLj}
 (1 - 2 \sigma_j) \frac{L_j}{M_j} \leq L_{j+1} \leq \frac{L_j}{M_j}.
\ee
This is motivated by \eqref{eq:sizetiL}.

\begin{lemma}
 The $L_j$ satisfy
 \begin{align}\label{eq:boundsLj}
  \E^{-4\sigma} \E^{-\frac{1}{99}} L 10^{-(j+1)(j+2)} \leq L_{j+1}  \leq L 10^{-(j+1)(j+2)}.
 \end{align}
\end{lemma}

\begin{proof}
 Recall from the last lemma that $\prod_{k=1}^{j}(1 - 2\sigma_k)\geq\E^{-4\sigma}$.
 An iteration of \eqref{eq:recurLj} shows
 $$
  \prod_{k=1}^{j}\frac{1 - 2 \sigma_k}{M_k+1} L_0 \leq L_{j+1} \leq \prod_{k=1}^{j} \frac{1}{M_k+1} L_0.
 $$
 Since
 $$
  1 \geq \prod_{k=1}^{j} \frac{M_k}{M_k+1} = \exp\left(-\sum_{k=1}^{j}\log\left(1 + \frac{1}{100^k}\right)\right)
  \geq \exp(-\frac{1}{99}),
 $$
 we have that \eqref{eq:productMj} implies the claim.
\end{proof}

We define $j_{max}$ by being the maximal $j$ such that
\be\label{eq:defjmax}
 \sigma_{j_{max}} L_{j_{max}} \geq 2 M_{j_{max}}
\ee
holds. This is needed in order that we can satisfy \eqref{eq:condsigmaLM}
in Theorem~\ref{thm:multistep1}. We have that

\begin{lemma}\label{lem:asymjmax}
 If $\sigma$ stays fixed, then $\delta_{j_{max}} \to \infty$ as $L \to \infty$.
 Furthermore,
 \be
  \delta_{j_{max}} L_{j_{max}} \geq \E^{-8\sigma} \E^{-\frac{1}{99}} L \delta
 \ee
\end{lemma}

\begin{proof}
 We observe that \eqref{eq:defjmax} only depends on $\sigma$ and $L$. Furthermore,
 if $L$ becomes large, the restriction becomes less and less restrictive.

 The second claim follows by \eqref{eq:lowbounddeltaj} and \eqref{eq:boundsLj}
\end{proof}

We will now start by exploiting the multiscale step stated in Theorem~\ref{thm:multistep1}.
We will show

\begin{theorem}\label{thm:multi1}
 Assume that
 \begin{align}
  \label{eq:cond1}  \frac{\sigma L}{M} & \geq 2 \\
  \label{eq:cond2} |\mathcal{E}| & \geq \E^{-\frac{1}{25} \sigma\delta} \\
  \label{eq:cond3} \frac{2^{17} \E^{12\sigma}}{\sigma^4} \cdot \left(\frac{N}{L}\right)^3 & \leq \E^{\frac{8}{25}\E^{-4\sigma} \sigma\delta}
 \end{align}
 hold and that $\{V(n)\}_{n=0}^{N-1}$ is $(\delta,\sigma,L,\mathcal{E})$-critical,
 then there is $\mathcal{E}_0\subseteq\mathcal{E}$ satisfying
 \be
  \frac{|\mathcal{E}_0|}{|\mathcal{E}|} \geq \exp\left(-\frac{25}{4} \frac{\E^{-\frac{8}{25}\sigma\delta}}{\sigma\delta \ln(50)}\right)
 \ee
 such that $\{V(n)\}_{n=0}^{N-1}$ is $(\delta_{j_{max}}, \sigma_{j_{max}}, L_{j_{max}}, \mathcal{E}_0)$-critical.
\end{theorem}

We will now start the proof of this theorem. The proof is based on
induction. First observe, that by the assumption that
$\{V(n)\}_{n=0}^{N-1}$ is $(\delta,\sigma,L,\mathcal{E})$-critical,
we have that
$\{V(n)\}_{n=0}^{N-1}$ is $(\delta_0,\sigma_0,L_0,\mathcal{E})$-critical.
This means that the base case is taken care of. The main problem with
applying induction is that the interval $\mathcal{E}$ will shrink
with the induction procedure, that is why we will need to do something
slightly more sophisticated.
This motivates the following definition:

\begin{definition}\label{def:acceptable}
 Given $\{V(n)\}_{n=0}^{N-1}$. A collection of intervals $\{\mathcal{E}_q\}_{q = 0}^{Q}$
 is called $(\sigma, \delta, L)$-acceptable if
 \begin{enumerate}
  \item For each $q$, we have that $\{V(n)\}_{n=0}^{N-1}$ is $(\sigma,\delta,L,\mathcal{E}_q)$-critical
  \item For $q, \ti{q}$, we have that $|\mathcal{E}_q| = |\mathcal{E}_{\ti{q}}|$.
  \item We have that
   \be
    |\mathcal{E}_q| \geq \E^{-\frac{1}{25} \sigma\delta}
   \ee
   for each $q$.
 \end{enumerate}
\end{definition}

We first observe that $\{\mathcal{E}\}$ is $(\sigma_0, \delta_0, L_0)$-acceptable,
since we assume criticality and \eqref{eq:cond2}.
This implies the following consequence of  Theorem~\ref{thm:multistep1}.

\begin{lemma}
 Given $\{V(n)\}_{n=0}^{N-1}$ and a collection of intervals $\{\mathcal{E}_q^j\}_{q = 0}^{Q_j}$
 is called $(\sigma_j, \delta_j, L_j)$-acceptable, then there exists a collection
 of intervals $\{\mathcal{E}^{j+1}_q\}_{q=0}^{Q_{j+1}}$ that is
 $(\sigma_{j+1}, \delta_{j+1}, L_{j+1})$-acceptable.
\end{lemma}

\begin{proof}
 All but condition (iii) of Definition~\ref{def:acceptable}
 are direct consequences of Theorem~\ref{thm:multistep1}.
 For (iii) observe that \eqref{eq:estisizeEq} implies that
 $$
  |\mathcal{E}_q^{j+1}| \geq \E^{-\sigma_j \delta_j}
 $$
 for any $q$. Now, observe that since $0 < \sigma_j\leq \frac{1}{4}$
 and $M_j \geq 100$, we have that
 $$
  \sigma_{j+1}\delta_{j+1} = \frac{1}{2}\sigma_j (1 - 2\sigma_j) M_j \delta_j
  \leq 25 \sigma_j\delta_j.
 $$
 So the claim follows.
\end{proof}

It remains to compare the size of
$$
 \bigcup_{q=0}^{Q_j} \mathcal{E}_q^j
 \text{ and }
 \bigcup_{q=0}^{Q_{j+1}} \mathcal{E}_q^{j+1}.
$$
For this, we will first need the following lemma.

\begin{lemma}
 Assume \eqref{eq:cond3}, then we have that
 \begin{align}
  \label{eq:cond3a} 10^{3(j+1)(j+2)} &\leq \E^{\frac{8}{25}\sigma_j\delta_j} \\
  \label{eq:cond3b} \frac{2^{17} \E^{12\sigma}}{\sigma^4} \cdot \left(\frac{N}{L}\right)^3 &\leq \E^{\frac{8}{25}\sigma_j\delta_j}.
 \end{align}
\end{lemma}

\begin{proof}
 Since $(j+1)(j+2) \leq 50^j$, these inequalities follow from
 $$
  10^3 \leq \E^{\frac{8}{25} \sigma\delta\E^{-4\sigma}}
  \text{ and }
  \frac{2^{17} \E^{12\sigma}}{\sigma^4} \cdot \left(\frac{N}{L}\right)^3 \leq \E^{\frac{8}{25}\E^{-4\sigma} \sigma\delta}
 $$
 By $N\geq L$ and $0 < \sigma\leq\frac{1}{4}$, we have that
 $$
  10^{3} \leq 2^{25} \leq \frac{2^{17} \E^{12\sigma}}{\sigma^4} \cdot \left(\frac{N}{L}\right)^3
 $$
 so both of the above equations follow from \eqref{eq:cond3}.
\end{proof}

The next lemma, will allow us to compare the size of an interval $\mathcal{E}_q^j$
to the size of the intervals $\mathcal{E}_p^{j+1}$ contained in $\mathcal{E}_q^j$.

\begin{lemma}
 We have that
 \be
  \frac{1}{|\mathcal{E}_q^{j}|} \cdot \left|\bigcup_{\mathcal{E}_p^{j+1} \subseteq \mathcal{E}_q^{j}}\mathcal{E}_p^{j+1}\right|
 \geq 1 - \E^{-\frac{8}{25}\sigma_j\delta_j}.
 \ee
\end{lemma}

\begin{proof}
 By \eqref{eq:boundsizecalQ}, we have that
 $$
  \left|\mathcal{E}_q^{j} \backslash \bigcup_{\mathcal{E}_p^{j+1} \subseteq \mathcal{E}_q^{j}}\mathcal{E}_p^{j+1}\right|
   \leq \frac{2^{17} 100^3}{\sigma_j^4} \frac{N^3}{L_j^3} \cdot |\mathcal{E}_s^{j+1}|.
 $$
 By construction, we have that \eqref{eq:estisizeEq} holds, that is
 $|\mathcal{E}_s^{j+1}| \leq \E^{-\sigma_j\delta_j}$.
 Hence, we obtain that
 $$
  \left|\mathcal{E}_q^{j} \backslash \bigcup_{\mathcal{E}_p^{j+1} \subseteq \mathcal{E}_q^{j}}\mathcal{E}_p^{j+1}\right|
   \leq \frac{2^{17} \E^{12\sigma}}{\sigma^4} \cdot \left(\frac{N}{L}\right)^3 \cdot 10^{3(j+1)(j+2)} \cdot \E^{-\sigma_j\delta_j}.
 $$
 Since, we have that $|\mathcal{E}_q^j| \geq \E^{-\frac{1}{25} \sigma_j\delta_j}$, we obtain
 that
 \begin{align*}
  \frac{1}{|\mathcal{E}_q^{j}|} \cdot \left|\bigcup_{\mathcal{E}_p^{j+1} \subseteq \mathcal{E}_q^{j}}\mathcal{E}_p^{j+1}\right|
  &\geq 1 - \frac{2^{17} \E^{12\sigma}}{\sigma^4} \cdot \left(\frac{N}{L}\right)^3 \cdot 10^{3(j+1)(j+2)} \cdot \E^{-\frac{24}{25} \sigma_j\delta_j}\\
  &\geq 1 - \E^{-\frac{8}{25}\sigma_j\delta_j},
 \end{align*}
 where we used \eqref{eq:cond3a} and \eqref{eq:cond3b}.
 This finishes the proof.
\end{proof}

We now come to

\begin{lemma}
 We have that
 \be
  \left|\bigcup_{q=0}^{Q_{j+1}} \mathcal{E}_q^{j+1} \right| \geq
  \left|\bigcup_{q=0}^{Q_j} \mathcal{E}_q^j\right| \cdot
  (1 - \E^{-\frac{8}{25}\sigma_j\delta_j}).
 \ee
\end{lemma}

\begin{proof}
 This is a consequence of the last lemma.
\end{proof}

\begin{proof}[Proof of Theorem~\ref{thm:multi1}]
 By the previous discussion, we can choose $\mathcal{E}_0$ such that
 $$
  |\mathcal{E}_0| \geq \prod_{j=1}^{\infty} (1 - \E^{-\frac{8}{25} \sigma_j\delta_j}) |\mathcal{E}|.
 $$
 Using \eqref{eq:lowbounddeltajsigmaj} and $\log(1 - x) \geq -2x$, we find that
 $$
  |\mathcal{E}_0| \geq \exp\left(-2\sum_{j=1}^{\infty} \E^{-\frac{8}{25}\sigma\delta\E^{-4\sigma} 50^j}\right)
   \geq \exp\left(-2 \frac{\E^{-\frac{8}{25}\sigma\delta}}{\frac{8}{25}\sigma\delta \ln(50)}\right),
 $$
 since $\sum_{j=1}^{\infty}\E^{-ta^j} \leq \frac{\E^{-t}}{\ln(a) t}$.
\end{proof}

%
%
%

\section{Proof of Theorem~\ref{thm:mainA}}
\label{sec:proofthmmainA}

We begin by observing that \eqref{eq:cond1thm} implies that, for $L$ large enough
$\{V(n)\}_{n =0}^{LK-1}$ is  $(\delta, \sigma, L, \mathcal{E})$-critical, $\delta = \gamma K$
in the sense of Definition~\ref{def:critical}. To see this, choose $k_j = j K$,
and $\mathcal{L}$ as the complement of the set in \eqref{eq:cond1thm}. The rest
follows. We now use the mechanism of the last two sections to improve the estimate.

\begin{lemma}
 $\{V(n)\}_{n=0}^{LK - 1}$ will be $(\hat{\delta}, \hat{\sigma}, \hat{L}, \widehat{\mathcal{E}})$-critical,
 where $\widehat{\mathcal{E}} \subseteq\mathcal{E}$ satisfies
 \be\label{eq:sizeEtoverE}
  \frac{|\widehat{\mathcal{E}}|}{|\mathcal{E}|} \geq \exp\left(-\frac{25}{4} \frac{\E^{-\frac{8}{25}\sigma\delta}}{\sigma\delta \ln(50)}\right)
 \ee
 and by Lemma~\ref{lem:asymjmax}, we have that
 \be
  \hat{\delta} \hat{L} \geq \E^{-8\sigma - \frac{1}{99}} \gamma K \cdot L.
 \ee
\end{lemma}

\begin{proof}
 Since $\{V_\omega(n)\}_{n=0}^{N_t - 1}$ is $(\delta,\sigma,L_t,\mathcal{E})$-critical,
 we now wish to apply Theorem~\ref{thm:multi1} to improve this estimate. In order to
 do this, we still have to ensure that \eqref{eq:cond1},\eqref{eq:cond2} \eqref{eq:cond3} hold.
 \eqref{eq:asA1} implies \eqref{eq:cond2}.
 \eqref{eq:cond3} is implied by \eqref{eq:asA2}. For \eqref{eq:cond1} observe
 that it is satisfied if $L$ is large enough.
\end{proof}

We now come to

\begin{lemma}
 We may choose the set $\widehat{\mathcal{E}}$ even so, that we have for every
 $\Lambda \subseteq [0, LK-1]$, that for every $E \in \widehat{\mathcal{E}}$, we have
 \be
  \dist(H, \sigma(H_{\Lambda})) \geq \E^{-\hat{\sigma} \hat{\delta}}.
 \ee
\end{lemma}

\begin{proof}
 This follows by an inspection of the argument of the last section.
\end{proof}

Now repeating the argument to obtain Green's function estimates as
done in Lemma~\ref{lem:greenimp}, we obtain the estimates
required by Lemma~\ref{lem:greentolyap}. Hence, we obtain that
\be
 \frac{1}{LK} \log\|A(E, LK)\| \geq \E^{-8\sigma}\E^{-\frac{1}{99}} \gamma KL - \frac{\sqrt{2}}{LK}
\ee
for $E \in \widehat{\mathcal{E}}$. This finishes the proof of Theorem~\ref{thm:mainA},
using that $\E^{-x} \geq 1 - x$ for $x \geq 1$.

%
%
%

\section{Proof of Theorem~\ref{thm:mainB}}
\label{sec:proofthmmainB}

We first need the following observation.

\begin{lemma}\label{lem:choiceomega}
 There exists $\omega\in\Omega$, such that the following properties
 hold
 \begin{enumerate}
  \item We have that
   \be
    L(E) \geq \limsup_{n\to\infty} \frac{1}{n} \log\|A_{\omega}(E,n)\|
   \ee
   for all $E$.
  \item There are sequences $N_t, L_t \to \infty$ such that
   $\{V_\omega(n)\}_{n=0}^{N_t - 1}$ is $(\delta,\sigma,L_t,\mathcal{E})$-critical
   and
   \be
    \lim_{t\to\infty} \frac{N_t}{L_t} = K.
   \ee
 \end{enumerate}
\end{lemma}

\begin{proof}
 We let $\Omega_{CS}$ be the set from Theorem~\ref{thm:cs}. This implies
 that property (i) holds as long as $\omega\in\Omega_{CS}$. Furthermore,
 we have that $\mu(\Omega_{CS}) = 1$.

 We let $\Omega_g$ be the complement of the set in \eqref{eq:cond2thma}.
 By Lemma~\ref{lem:ergodic1}, we can find a set $\widetilde{\Omega}$
 with $\mu(\widetilde{\Omega}) > 0$, and for
 each $\omega\in\widetilde{\Omega}$ sequences $N_t, L_t \to \infty$ such that
 property (ii) holds.

 So we have that $\Omega_0 \cap \widetilde{\Omega}$ is non-empty and
 by choosing $\omega\in\Omega_0\cap\widetilde{\Omega}$, we are done.
\end{proof}

We now fix $\omega$ as in the last lemma, and abbreviate
\be
 V(n) = V_{\omega}(n).
\ee
The claim now follows by applying Theorem~\ref{thm:mainA} (more exactly
the quantitative version) to $\{V(n)\}_{n=0}^{N_t - 1}$. Giving more
details, we obtain a sequence of sets $\mathcal{E}_t$, satisfying
$$
 |\mathcal{E}_t| \geq (1 - \E^{-\frac{8}{25} \sigma \gamma K}) |\mathcal{E}|
$$
and for $E \in \mathcal{E}_t$, we have
$$
 \frac{1}{N_t} \log\|A(E, N_t)\| \geq \E^{-8\sigma}\E^{-\frac{1}{99}} \gamma + o(1)
$$
as $t \to \infty$. Hence, we have that
$$
 L(E) \geq \E^{-8\sigma}\E^{-\frac{1}{99}} \gamma
$$
for
$$
 E \in \mathfrak{E} = \bigcap_{s \geq 1} \bigcup_{t \geq s} \mathcal{E}_t.
$$
We have that

\begin{lemma}
 The set
 $\mathfrak{E} = \bigcap_{s \geq 1} \bigcup_{t \geq s} \mathcal{E}_t$
 has measure
 \be
  |\mathfrak{E}| \geq (1 - \E^{-\frac{8}{25} \sigma \gamma K}) |\mathcal{E}|.
 \ee
\end{lemma}

\begin{proof}
 Let $\mathfrak{E}_s = \bigcup_{t \geq s} \mathcal{E}_t$. We have that
 $\mathfrak{E}_{s+1} \subseteq \mathfrak{E}_s$ and
 $|\mathfrak{E}_s| \geq (1 - \E^{-\frac{8}{25} \sigma \gamma K}) |\mathcal{E}|$.
 This implies the claim, since $\mathfrak{E}_s \subseteq \mathcal{E}$
 with $|\mathcal{E}| < \infty$.
\end{proof}

This finishes the proof of Theorem~\ref{thm:mainB}.

%
%
%

\section{The initial condition at large coupling: Proof of Proposition~\ref{prop:initlarge}}
\label{sec:initlarge}

In this section, we will discuss how our initial
conditions can be verified for large $\lambda$.
We let $(\Omega,\mu)$ be a probability space and
$T: \Omega\to\Omega$ an ergodic transformation
(measure preserving is enough for the purpose of this section).
Given a function $f: \Omega\to\R$ and $\lambda > 0$,
we introduce our potential by
\be
 V_{\omega, \lambda}(n) = \lambda f(T^n \omega),
\ee
where $\omega \in \Omega$. We will assume that $f: \Omega \to \R$
is non-degenerate in the sense of Definition~\ref{def:nondegenerate}.
That is, there are $F, \alpha > 0$ such that for all $E \in \R$
\be\label{eq:condflat}
 \mu(\{\omega\in\Omega:\quad |f(x) - E| \leq \eps\}) \leq F \eps^{\alpha}.
\ee
Before coming to the proof of Proposition~\ref{prop:initlarge},
We first recall the Combes--Thomas estimate (see \cite{ct})

\begin{lemma}\label{lem:ct}
 Let $\Lambda \subseteq \Z$, $V: \Lambda \to \R$ be a bounded sequence, and
 $H: \ell^2(\Lambda) \to \ell^2(\Lambda)$ be defined by its action on $u \in \ell^2(\Lambda)$
 by
 \be
  Hu(n) = u(n+1) + u(n-1) + V(n) u(n)
 \ee
 for $n \in \Lambda$ (where we set $u(n) = 0$ for $n\notin\Lambda$). Assume
 that $\dist(\sigma(H), E) > \delta$.
 Let
 \be
  \gamma = \frac{1}{2} \log(1 + \frac{\delta}{4}),\quad K = \frac{1}{\gamma}\log(\frac{4}{3\delta}).
 \ee
 Then for $k,l \in \Lambda$, $|k - l| \geq K$, the estimate
 \be
  |G(E,k,l)| \leq \frac{1}{2} \E^{-\gamma |k - l|}
 \ee
 holds.
\end{lemma}

We start by observing the following lemma.

\begin{lemma}
 Let $f$ be a non-flat function, $K\geq 1$, $B >0$.
 Then for $E \in \R$, the set
 \be
  A_{K,B}(E) = \{\omega\in\Omega:\quad |f(T^k\omega) - E| \geq B,\, k=0,\dots,K-1\}
 \ee
 has measure
 \be
  \mu(A_{K,B}(E)) \geq 1 - B^{\alpha} F K.
 \ee
\end{lemma}

\begin{proof}
 By \eqref{eq:condflat}, the set
 $$
  A_{B}(E) = \{\omega\in\Omega:\quad |f(\omega) - E| < B\}
 $$
 has measure $\mu(A_{B}(E)) \leq B^{\alpha} F$. Since
 $$
  A = \Omega \backslash \left(\bigcup_{k=0}^{K-1} T^{-k} A_{B}(E)\right)
 $$
 the claim follows and $T: \Omega\to\Omega$ being measure preserving.
\end{proof}

This implies

\begin{lemma}\label{lem:longstretch}
 Let $(\Omega,\mu,T,f)$ be as above. Let $E_0 \in \R$ and $\sigma > 0$. Introduce
 \be
  K(\lambda) = \left\lfloor \frac{\sigma \lambda^{\alpha/2}}{F} \right\rfloor.
 \ee
 Then there is a set $A$ of measure $\mu(A) \geq 1 - \frac{1}{2} \sigma$ such that
 for $\omega \in A$, we have that
 \be
  |\lambda f(T^k \omega) - E_0| > \sqrt{\lambda},
 \ee
 for $k =0 , \dots, K(\lambda)$.
\end{lemma}

\begin{proof}
 Letting $B = \frac{1}{\sqrt{\lambda}}$ in the last lemma, we obtain
 that the set $A_{K,B}(\frac{1}{\lambda} E_0)$ has measure
 $\mu(A_{K,B}(E)) \geq 1 - \frac{F K}{\lambda^{\alpha/2}}$.
 We have $\mu(A_{K,B}(E)) \geq 1 - \frac{1}{2}\sigma$ as long as
 $\frac{FK}{\lambda^{\alpha/2}} \leq \frac{1}{2}\sigma$.
 Hence the claim follows.
\end{proof}

We are now ready for

\begin{proof}[Proof of Proposition~\ref{prop:initlarge}]
 By Lemma~\ref{lem:longstretch}, we obtain $A \subseteq \Omega$
 of measure $\mu(A) \geq 1 - \frac{1}{2}\sigma$ and such that
 $$
  \mathrm{dist}(\mathcal{E}, \sigma(H_{\omega, [0,M-1]})) \geq \sqrt{\lambda} - 3 > \frac{1}{2} \sqrt{\lambda}
 $$
 for $\omega \in A$ (Here we used $\lambda > 36$).  We choose $M = 2K - 2$.
 We may thus  apply the Combes--Thomas estimate (Lemma~\ref{lem:ct})
 to obtain that
 $$
  |G_{T^{-1} \omega,[1,2K-2]}(E, K, l)| \leq \frac{1}{2} \E^{-\gamma M}
 $$
 for $l \in \{1,2K-1\}$. Hence, we see that
 $[1, 2K -2]$ is $(\gamma, \mathcal{E})$-good for $H_{T^{-1}\omega}$
 in the sense of Definition~\ref{def:good}. This finishes the proof.
\end{proof}

%
%

\section{The Pastur--Figotin formalism and proof of Proposition~\ref{prop:random}}
\label{sec:pf}

In this section we will prove Proposition~\ref{prop:random},
for this we develop the Pastur--Figotin formalism
from \cite{pf} as it was improved by Chulaevsky and
Spencer in \cite{chsp} and later in Bourgain and Schlag \cite{bs},
and then use to it to prove large deviation estimates for
matrix elements of the Green's function. We will denote by
$H$ the operator defined in \eqref{eq:defHulV}

We will begin by introducing Pr\"ufer
variables. Define $\rho(n)$, $\varphi(n)$ for a solution
$u$ of $H u = 2 \cos(\kappa) u$ by
\be\label{eq:defphirho} \begin{split}
 \rho(n) \sin(\varphi(n)) & = \sin(\kappa) u(n-1) \\
 \rho(n) \cos(\varphi(n)) & = u(n) - \cos(\kappa) u(n-1).
\end{split} \ee
This implies the following lemma, after a bit of computation.

\begin{lemma}\label{lem:rhotosol}
 We let $u$ be the solution of $Hu = 2 \cos(\kappa) u$, with initial conditions
 \be\label{eq:initu}
  u(0)= \frac{\sin(\theta)}{\sin(\kappa)} \rho(1),\quad
  u(1)= \cos(\theta) - \frac{\cos(\kappa)}{\sin(\kappa)} \sin(\theta) \rho(1).
 \ee
 We have that
 \begin{align}
  \label{eq:uleqrho}   \min(|u(n-1)|, |u(n)|) &\leq \frac{1}{\sqrt{1 - |\cos(\kappa)|}} \rho_\theta(n)^2, \\
  \label{eq:ugeqrho}   \max(|u(n-1)|, |u(n)|) &\geq \frac{1}{2} \rho_\theta(n).
 \end{align}
\end{lemma}

In the following, we will fix $\kappa \in (0, \pi) \backslash \{\pi/2\}$
and let $\rho_\theta$, $\varphi_\theta$ denote the Pr\"ufer variables
with initial condition \eqref{eq:initu}.
We will prove the following proposition in the next section.

\begin{proposition}\label{prop:ldt}
 Assume the following inequalities
 \begin{align}
  \label{eq:condN1}       N &\geq \frac{344 \cdot \sigma_2}{|\sin(\kappa) \cos(\kappa)| \cdot \min(1, 2|\cos(\kappa)^2 - \sin(\kappa)^2|)} \\
  \label{eq:condlam1}  \lam &\leq |\sin(\kappa)| \min\left( \frac{\sigma_2}{7000}, \frac{|\sin(\kappa) \cos(\kappa)| \cdot \min(1, 2|\cos(\kappa)^2 - \sin(\kappa)^2|)}{1032 \cdot \sigma_2}\right)
 \end{align}
 Introduce
 \be
  \gamma_1 = \frac{\sigma_2 \lambda^2}{8 \sin(\kappa)^2}.
 \ee
 We have that
 \begin{align}\label{eq:ldtrhoN}
  \nu^{\otimes N}&\left(\{\ul{V}\quad |\frac{1}{N}\log(\rho_N(\theta)) - \gamma_1| \geq \frac{1}{6} \gamma_1\}\right)\\
  \nn &\leq \frac{2400}{N} \cdot \frac{\sigma_4}{(\sigma_2)^2} + 3 \E^{-\frac{\gamma_1^2 N}{80000}}.
 \end{align}
\end{proposition}

We now begin deriving consequences of the last proposition.

\begin{lemma}
 Assume \eqref{eq:randK}, \eqref{eq:randlam}, and \eqref{eq:randlam2K}
 then
 \begin{align}
 \label{eq:pfdet1} \nu^{\otimes 2K}\left(\{\ul{V}:\quad \sup_{M\in\{2K-3, 2K-2\}} |\det(H_{\ul{V}, [1, M]} - E)|
    \leq \frac{\E^{\frac{5}{3} \gamma_1 K}}{\sqrt{1 - |\cos(\kappa)|}} \} \right)
    & \leq  \frac{1}{48}\\
 \label{eq:pfdet2} \nu^{\otimes 2K}\left(\{\ul{V}:\quad |\det(H_{\ul{V}, [1, K-1]} - E)|
   \geq \frac{1}{2} \E^{\frac{7}{6} \gamma_1  K}\}\right)
     & \leq  \frac{1}{48}\\
 \label{eq:pfdet3} \nu^{\otimes 2K}\left(\{\ul{V}:\quad \sup_{M\in\{2K-3, 2K-2\}} |\det(H_{\ul{V}, [K+1, M]} - E)|
    \geq \frac{1}{2} \E^{\frac{7}{6} \gamma_1  K}\}\right)
    & \leq  \frac{1}{48}.
 \end{align}
 hold.
\end{lemma}

\begin{proof}
 Observe that \eqref{eq:randK} implies \eqref{eq:condN1} with $N = K/2$.
 We need to make a few observations. First, if we choose $\theta$ depending
 on $v_0$, we can still apply the above estimates to $\ul{V}'$ such that
 $\ul{V} = (v_0,\ul{V}')$. Next, we may choose $\theta=\theta(v_0)$ in
 such a way that
 $$
  u(n) = \det(H_{\ul{V}, [1, n]} - E)
 $$
 for $n\geq 1$. We do this and obtain by \eqref{eq:uleqrho} that
 $$
  \sup(|\det(H_{\ul{V}, [1, 2K-3]} - E)|, |\det(H_{\ul{V}, [1, 2K-2]} - E)|) \leq \frac{1}{\sqrt{2 (1 - |\cos(\kappa)|)}} \rho_{\theta}(2K-2).
 $$
 Hence, we apply \eqref{eq:ldtrhoN} with $N =  2K - 3$ for \eqref{eq:pfdet1}.
 The claim now follows by a sequence of computations. \eqref{eq:pfdet2} and
 \eqref{eq:pfdet3} are similar, but we need $N = K -3$. So since, we assume
 $K \geq 6$, we have $N\geq K/2$, which is exactly our assumption.
\end{proof}

We need the following lemma

\begin{lemma}
 Assume that the potential $V(n)$ is bounded by $C > 0$ and
 $H_{[0,M-1]}$ acts on $\ell^2([0,M-1])$, then for $|E| \leq 2 + C$
 we have that
 \be\label{eq:normresolvpf}
  \|(H_{[0,M]} - E)^{-1}\|_{\mathrm{HS}} \leq
  \frac{M (4 + 2 C)^{M/2}}{|\det(H_{[0,M]} - E)|}.
 \ee
\end{lemma}

\begin{proof}
 By Cramer's rule, we have that
 \begin{align*}
  \|(H_{[0,M]} - E)\|_{\mathrm{HS}}^2 = &\frac{2}{|\det(H_{[0,M]} - E)|^2} \\
   &\quad\left(\sum_{0 \leq j < k \leq M} |\det(H_{[0,j-1]} - E)|^2\cdot |\det(H_{[k+1, M]} - E)|^2\right).
 \end{align*}
 By Hadamard's inequality, we have $|\det(H_{[x,y]} - E)|^2 \leq \prod_{i=x}^{y} (2 + |V(i) - E|^2) \leq (4 + 2 C)^{y-x+1}$.
 Thus
 $$
  \|(H_{[0,M]} - E)^{-1}\|_{\mathrm{HS}}^2 = \frac{M^2 (4 + 2 C)^M}{|\det(H_{[0,M]} - E)|^2}.
 $$
 This implies the claim.
\end{proof}

Now we are ready for

\begin{proof}[Proof of Proposition~\ref{prop:random}]
 By \eqref{eq:pfdet1}, we can choose $M\in\{2K -3, 2K-2\}$ and $\ul{V}$
 in a set of measure $1 - \frac{1}{48}$ such that
 $$
  |\det(H_{\ul{V}, [1, M]} - E)| \geq  \frac{1}{\sqrt{1 - |\cos(\kappa)|}}\E^{\frac{5}{3} \gamma_1 \cdot K}.
 $$
 By Cramer's rule, we have that
 $$
  |G_{\ul{V}, [1, M]}(E, 1, K)| = \frac{|\det(H_{\ul{V}, [K+1, M]} - E)|}{|\det(H_{\ul{V}, [1, M]} - E)|}
 $$	
 and
 $$
  |G_{\ul{V}, [1, M]}(E, M, K)| = \frac{|\det(H_{\ul{V}, [1, K-1]} - E)|}{|\det(H_{\ul{V}, [1, M]} - E)|}.
 $$
 These imply the first two inequalities. The third follows from \eqref{eq:normresolvpf}.
\end{proof}

%
%
%

\section{Proof of Proposition~\ref{prop:ldt}}
\label{sec:ldtrandom}

Let $\varphi(n)$ and $\rho(n)$ be as defined in \eqref{eq:defphirho}. Introduce
\be
 \zeta(n) = \E^{2 \I \varphi(n)},\quad
 \mu = \E^{2 \I \kappa}.
\ee
We have that (see \cite{kls}, \cite{krs})

\begin{lemma}
 The next equations hold
 \begin{align}
  \label{eq:zetan1} \zeta(n+1) &= \mu \zeta(n) + \frac{\I \lambda}{2} \frac{V(n)}{\sin(\kappa)}
  \frac{(\mu \zeta(n) - 1)^2}{1 - \frac{\I \lambda}{2} \frac{V(n)}{\sin(\kappa)}(\mu \zeta(n) - 1)}, \\
  \label{eq:rhon1} \frac{\rho(n+1)^2}{\rho(n)^2} & = 1 + \frac{\lambda}{2}  \frac{V(n)}{\sin(\kappa)} (\zeta(n)\mu - \ol{\zeta(n)\mu}) \\
  \nn &\qquad+ \frac{\lambda^2}{2} \left(\frac{V(n)}{\sin(\kappa)}\right)^2 (\zeta(n) \mu - 2 + \ol{\zeta(n) \mu}).
 \end{align}
 Here $\ol{z}$ denotes the complex conjugate.
\end{lemma}

We start by verifying an inequality

\begin{lemma}
 Assume the inequalities \eqref{eq:condN1} and \eqref{eq:condlam1}, then
 \be\label{eq:condsigma2mu}
  \frac{\sigma_2}{\min(|1-\mu|,|1-\mu^2|)} \left(\frac{2}{N} + \frac{6\lam}{|\sin(\kappa)|}\right)\leq\frac{1}{172}
 \ee
 holds.
\end{lemma}

\begin{proof}
 Observe that
 \begin{align*}
  |1 - \mu| & \geq |\im(\mu)| = |\sin(2\kappa)| = 2 |\sin(\kappa)| |\cos(\kappa)| \\
  |1 - \mu^2| & \geq |\im(\mu^2)| = |\sin(4\kappa)| = 4 |\sin(\kappa)| |\cos(\kappa)| |\cos(\kappa)^2 - \sin(\kappa)^2|.
 \end{align*}
 Now the claim is a quick computation.
\end{proof}

We are now ready for

\begin{lemma}
 Assume \eqref{eq:condlam1} and \eqref{eq:condN1}, then
 \be\label{eq:boundsumzeta}
  |\sum_{n = 1}^{N} \zeta(n)| \leq \frac{1}{172} \frac{N}{\sigma_2},\quad
  |\sum_{n = 1}^{N} \zeta(n)^2| \leq \frac{1}{172} \frac{N}{\sigma_2}
 \ee
 hold.
\end{lemma}

\begin{proof}
 First, \eqref{eq:zetan1} implies that $|\zeta(n+1) - \mu\zeta(n)|\leq 3\frac{\lam}{|\sin(\kappa)|}$,
 and since
 $$
  \zeta(n+1)^2 - \mu^2 \zeta(n) = \zeta(n+1) (\zeta(n+1) - \mu \zeta(n)) + \mu \zeta(n) (\zeta(n+1) - \mu\zeta(n)),
 $$
 also $|\zeta(n+1)^2 - \mu^2\zeta(n)^2|\leq 6\frac{\lam}{|\sin(\kappa)|}$.
 Hence from $\sum_{n = 1}^{N} \zeta(n) = \zeta(1) + \sum_{n = 1}^{N-1} \zeta(n+1)$, we
 obtain
 $$
  \left| (1 - \mu) \sum_{n = 1}^{N} \zeta(n) + \zeta(1) + \mu \zeta(N)\right| \leq 3 N \frac{\lam}{|\sin\kappa|}
 $$
 This implies \eqref{eq:boundsumzeta} by the last lemma.
\end{proof}

We will now suppose that for some $\theta \in [0,\pi)$, we consider
the solution to \eqref{eq:zetan1} and \eqref{eq:rhon1} satisfying
the initial conditions
\be
 \zeta(0) = \E^{2 i \theta}, \quad \rho(0) = 1.
\ee
In order to highlight the dependence on $\theta$, we will sometimes
write $\zeta_{\theta}(n)$ and $\rho_{\theta}(n)$.
Introduce the following terms
\begin{align}
 \mathcal{F}_1(\theta, \ul{V}, N) &=
  \frac{\lambda^2}{8 N \sin(\kappa)^2} \sum_{n=1}^{N} V(n)^2\\
 \mathcal{F}_2(\theta, \ul{V}, N) &=
  \frac{\lambda}{4 N \sin(\kappa)} \sum_{n=1}^{N} V(n) (\zeta_{\theta}(n) \mu - \ol{\zeta_{\theta}(n)\mu}) \\
 \mathcal{F}_3(\theta, \ul{V}, N) &=
  - \frac{\lambda^2}{8 N \sin(\kappa)^2} \sum_{n=1}^{N} V(n)^2 (\zeta_{\theta}(n) \mu + \ol{\zeta_{\theta}(n)\mu}) \\
 \mathcal{F}_4(\theta, \ul{V}, N) &=
  \frac{\lambda^2}{16 N\sin(\kappa)^2} \sum_{n=1}^{N} V(n)^2 ((\zeta_{\theta}(n)\mu)^2 + (\ol{\zeta_{\theta}(n)\mu})^2).
\end{align}
We furthermore introduce
\be
 \mathcal{F}(\theta, \ul{V}, N) = \mathcal{F}_1(\theta, \ul{V}, N) + \dots + \mathcal{F}_4(\theta, \ul{V}, N).
\ee
We obtain the following lemma

\begin{lemma}
 Assume \eqref{eq:condlam1}. For any $\theta\in[0,\pi)$, we have that
 \begin{align}\label{eq:diflogrhocalF}
  \Bigg|\frac{1}{N} \log(\rho_{\theta}(N)) - \mathcal{F}(\theta, \ul{V}, N)) \Bigg|
 \leq  \frac{\gamma_1}{12}
 \end{align}
\end{lemma}

\begin{proof}
 Let
 $$
  x(n) = \frac{\lambda V(n)}{2 \sin(\kappa)} (\zeta(n)\mu - \ol{\zeta(n)\mu})
  + \frac{(\lambda V(n))^2}{2 \sin(\kappa)^2} (\zeta(n) \mu - 2 + \ol{\zeta(n) \mu}),
 $$
 so $|x(n)| \leq \frac{3 \lambda}{|\sin(\kappa)|} \leq \frac{1}{2}$ and by \eqref{eq:rhon1}
 $\frac{\rho(n+1)^2}{\rho(n)^2} = 1 + x(n)$.
 Since $\rho(1) = 1$, we have that
 $\log(\rho_N(\theta)) = \sum_{n=1}^{N} \log(1 + x(n))$.
 Using that
 $|\log(1 + x) - x + \frac{x^2}{2}| \leq \frac{|x|^3}{3(1 - x)^3}$,
 and $|x(n)| \leq \frac{1}{2}$, we find
 $$
  |\log(1 + x) - x + \frac{x^2}{2}| \leq \frac{8}{3} |x|^3,
 $$
 and the claim follows, by expanding the terms and comparing them.
\end{proof}

We next have that

\begin{lemma}
 We have that
 \be
  \nu^{\otimes N}(\{\ul{V}:\quad |\mathcal{F}_1 - \gamma_1| \geq \frac{1}{48} \gamma_1)
  \leq \frac{2400}{N} \cdot \frac{\sigma_4}{(\sigma_2)^2}.
 \ee
\end{lemma}

\begin{proof}
 One can compute that $\int \mathcal{F}_1 d\nu^{\otimes N} = \frac{\lambda^2 \sigma_2}{8 \sin(\kappa)^2}$
 and
 $$
  \int \left(\mathcal{F}_1 - \frac{\lambda^2 \sigma_2}{8}\right)^2 d\nu^{\otimes N} =\frac{1}{N} \frac{\lambda^4 \sigma_4}{64 \sin(\kappa)^4}.
 $$
 The claim then follows by Chebychev's inequality.
\end{proof}

We will need the following result, which is Azuma's Inequality
(Theorem 7.2.1. in Alon and Spencer \cite{as}).

\begin{theorem}\label{thm:azuma}
 Let $X_1, X_2, \dots, X_N: [-1,1]^N \to \R$ be functions satisfying
 the following three conditions:
 \begin{enumerate}
  \item $X_n$ only depends on $V_1, \dots, V_n$.
  \item $|X_n| \leq 1$.
  \item $\int_{[-1,1]} X_n(V_1, \dots, V_{n-1}, V_n) d\nu(V_{n}) = 0$
   for any $V_1, \dots, V_{n-1} \in [-1,1]$.
 \end{enumerate}
 Then
 \be
  \nu^{\otimes N}(\{\ul{V} \in [-1,1]^N:\quad \left|\sum_{n=1}^N X_n(\ul{V}) \right| \geq \lambda \sqrt{N}\})
   \leq \E^{-\frac{1}{2} \lambda^2}.
 \ee
\end{theorem}

We note that properties (i) - (iii) imply that $X_1, \dots, X_N$ form
a martingale.

\begin{lemma}\label{lem:ldtf2}
 We have that
 \be
  \nu^{\otimes N}(\{\ul{V} \in [-1,1]^N:\quad |\mathcal{F}_2| \geq \frac{1}{48} \gamma_1\})
   \leq \E^{-\frac{1}{4800} \gamma^2 N}.
 \ee
\end{lemma}

\begin{proof}
 In view of the definition of $\mathcal{F}_2$, we introduce
 $$
  X_n = \frac{\lambda}{4} \frac{V(n)}{\sin(\kappa)} (\zeta_\theta(n) \mu - \ol{\zeta_\theta(n) \mu}),
 $$
 so that $\mathcal{F}_2=\frac{1}{N} \sum_{n=1}^{N} X_n$.
 By \eqref{eq:zetan1}, we have that $\zeta_\theta(n) \mu - \ol{\zeta_\theta(n) \mu}$
 only depends on $V(1), \dots, V(n-1)$. Hence, we see that $\int X_n d\nu(V_n) = 0$,
 since $\int xd\nu = 0$. The other conditions of Theorem~\ref{thm:azuma}
 are straightforward to check, and the result follows.
\end{proof}

\begin{lemma}
 We have that
 \be
  \nu^{\otimes N}(\{\ul{V} \in [-1,1]^N:\quad |\mathcal{F}_{j}| \geq \frac{1}{48} \gamma_1\})
   \leq \E^{-\frac{1}{80000} \gamma_1 ^2 N}.
 \ee
 for $j = 3,4$
\end{lemma}

\begin{proof}
 Introduce
 $$
  F_3 = - \frac{\lambda^2}{8 N} \sum_{n=1}^{N} \left(\frac{V(n)}{\sin(\kappa)}\right)^2 \zeta(n) \mu
 $$
 so that $\mathcal{F}_3 = F_3 + \ol{F_3}$. Now, decompose
 $$
  F_3 = - \frac{\lambda^2}{8 N \sin(\kappa)^2} \sum_{n=1}^{N} (V(n)^2 -\sigma_2^2) \zeta(n) \mu
  - \frac{\lambda^2 \sigma_2^2 \mu}{8 N \sin(\kappa)^2} \sum_{n=1}^{N} \zeta(n).
 $$
 We first observe that by \eqref{eq:boundsumzeta}, we have that
 \begin{align*}
  \left|\frac{\lambda^2 \sigma_2^2 \mu}{8 N \sin(\kappa)^2} \sum_{n=1}^{N} \zeta(n)\right|
   &\leq \frac{1}{172} \frac{\lambda^2 \sigma_2 }{8 \sin(\kappa)^2 \cdot} = \frac{\gamma_1}{172}.
 \end{align*}
 Introduce $X_n = \frac{\lambda^2}{8} (V_n^2 -\sigma_2^2) \zeta(n)$, such that
 $$
  |\mathcal{F}_3 - \frac{1}{N} \sum_{n=1}^{N} (X_n + \ol{X_n})|
  \leq \frac{\gamma_1}{96}.
 $$
 Next, we observe that $X_n$ obeys the condition of Theorem~\ref{thm:azuma},
 and we can conclude that
 $$
  \nu^{\otimes N}(\{\ul{V} \in [-1,1]^N:\quad \left|\frac{1}{N}\sum_{n=1}^{N} X_n\right| \geq \frac{\gamma_1}{172}\})
   \leq \E^{-\frac{1}{2} \left(\frac{\gamma_1}{172}\right)^2 N}.
 $$
 This finishes the proof of the first statement.
 A similar estimate works for $\mathcal{F}_4$.
\end{proof}

By the last sequence of lemma, we have shown Proposition~\ref{prop:ldt}.

%
%
%

\section{A variant of the multiscale step}
\label{sec:multistep2}

In this section, we will discuss a variant of the argument of Section~\ref{sec:multistep1}.
The main idea is instead of eliminating energies $E$ as done in Lemma~\ref{lem:energyelem},
we will assume a Wegner type estimate. In particular, this means that the results
of this section will be very close in spirit to the ones used for random Schr\"odinger
operators.

\begin{theorem}\label{thm:multistep2}
 Assume that $\{V(n)\}_{n=0}^{N-1}$ is $(\delta,\sigma,L,\mathcal{E})$-critical,
 $M \geq 3$ and \eqref{eq:condsigmaLM} (that is $\frac{\sigma L}{M} \geq 2$).
 Furthermore assume that
 \begin{align}\label{eq:condtinyresonant}
  \#\{0 \leq l \leq L:&\quad \{V(n)\}_{n=0}^{N-1}\text{ is }([k_l, k_l + \frac{16 N (M+1)}{\sigma L}], \mathcal{E}, 2 \E^{-\sigma\delta})
   \text{ resonant}\}\\
   \nn& \leq \frac{\sigma}{4} (1 - 2 \sigma) \frac{L}{M+1}.
 \end{align}
 Then $\{V(n)\}_{n=0}^{N-1}$ is also $(\ti{\delta}, \ti{\sigma}, \ti{L}, \mathcal{E})$-critical,
 with the quantities defined as in Theorem~\ref{thm:multistep1}.
\end{theorem}

The proof of this theorem parallels the proof of Theorem~\ref{thm:multistep1}.
We define $\ti{k}_j$ as in \eqref{eq:deftik1}, \eqref{eq:deftik2}, whose
properties stay the same. In particular $\ti{L}$ satisfies
\be
 (1 - 2 \sigma) \frac{L}{M+1} \leq \ti{L} \leq \frac{L}{M+1},
\ee
by the same argument as was used to show \eqref{eq:lowboundtiL}.

Instead of using Lemma~\ref{lem:energyelem} to find
the set $\mathfrak{L}$ of good indices for $\ti{k}_l$,
we will proceed differently. Denote by $l\notin\widetilde{\mathcal{L}}_0$
the set defined in \eqref{eq:defticalL0}, and the estimate
\eqref{eq:sizetiL0} on its size still holds.
We now let
$$
 \widetilde{\mathcal{L}}_1 = \{0 \leq l \leq L:\quad \{V(n)\}_{n=0}^{N-1}\text{ is }([k_l, k_l + \frac{16 N (M+1)}{\sigma L}], \mathcal{E}, 2 \E^{-\sigma\delta}) \text{ resonant}\},
$$
with \eqref{eq:condtinyresonant} now saying $\#\widetilde{\mathcal{L}}_1 \leq \frac{\ti{\sigma}}{2} \ti{L}$
after a short computation. Hence, we introduce
$$
 \mathfrak{L} = \widetilde{\mathcal{L}}_0 \cup \widetilde{\mathcal{L}}_1,
$$
which satisfies $\#\mathfrak{L} \leq \ti{\sigma} \ti{L}$. Now, we
are ready for.

\begin{proof}[Proof of Theorem~\ref{thm:multistep2}]
 One then sees that Lemma~\ref{lem:greenimp} still applies and the proof
 is finished in a similar fashion as the one of Theorem~\ref{thm:multistep1}.
\end{proof}

%
%
%

\section{Adaptation of the multiscale argument}
\label{sec:multiscale2}

In this section $\sigma_j, \delta_j, L_j, M_j$ denote the same constants as
in Section~\ref{sec:multiscale1}. We introduce
\be
 \eps_j = 3 \E^{-\sigma_j \delta_j}.
\ee
We have the following lemma. We note that the choice of intervals,
comes from \eqref{eq:condtinyresonant}.

\begin{lemma}
 Introduce the interval
 \be\label{eq:defEJE}
  \mathcal{E} = [E - 2 \E^{-\sigma_J \delta_J},E + 2 \E^{-\eps_J \delta_J}]
 \ee
 Then we have that
 \be
  \mathcal{E} + 2 [-\E^{-\sigma_j\delta_j}, 2 \E^{-\sigma_j\delta_j}]
  \subseteq [E - \eps_j,E + \eps_j],
 \ee
 for $0 \leq j \leq j_0 = j_0(J)$ and $\lim_{J\to\infty} j_0(J) = \infty$.
\end{lemma}

\begin{proof}
 This follows from the fact that the sequence $\sigma_j\delta_j \gtrsim 10^{j^2}$.
\end{proof}

We need the following lemma, one a numerical constant arising
in Theorem~\ref{thm:multistep2}.

\begin{lemma}
 Let $K_j$ be the length required by Theorem~\ref{thm:multistep2},
 for $(\delta_j, \sigma_j, L_j, \mathcal{E})$, then
 \be\label{eq:estiKj}
  K_j \leq \hat{K} \left(10^{(j+1)(j+2)}\right)^3, \quad \hat{K} = \E^{4\sigma} \E^{\frac{1}{99}} \frac{N}{L_0}.
 \ee
\end{lemma}

\begin{proof}
 First, observe that the $K_j$'s is given by $K_j = \frac{16 N (M_j + 1)}{\sigma_l L_j}$.
 By \eqref{eq:boundsLj}, we obtain that
 $$
  \frac{N}{L_j} \leq \E^{4\sigma} \E^{\frac{1}{99}} 10^{j(j+1)} \frac{N}{L_0}.
 $$
 By \eqref{eq:productMj}, we have that $M_j = 10^{(j+1)(j+2)}$, and since $j \leq j^2$,
 the result follows.
\end{proof}

We furthermore collect the following lemma, which is similar
to Lemma~\ref{lem:choiceomega}

\begin{lemma}\label{lem:choiceomega2}
 Assume \eqref{eq:aswegner1a}.
 There exists $\omega\in\Omega$, such that the following properties
 hold
 \begin{enumerate}
  \item We have that
   \be
    L(E) \geq \limsup_{n\to\infty} \frac{1}{n} \log\|A_{\omega}(E,n)\|
   \ee
   for all $E$.
  \item There is $N_0 \geq 1$ such that for $N \geq N_0$, we have that
   $\{V_\omega(n)\}_{n=0}^{N - 1}$ is $(\delta,\sigma,\lceil \frac{N}{K} - 1 \rceil,\mathcal{E})$-critical
  \item For $j \geq 1$, there is $N_j \geq 1$ such that for $N \geq N_j$,
   we have that
   \begin{align}\label{eq:numresonant}
    \#\{0\leq l\leq \frac{N}{K_0} :\quad &\{V_{\omega}\}_{n=0}^{N-1}\text{ is }([l K_0, l K_0 + K_j], \{E\}, \eps_j)\text{ resonant}\} \\
    \nn & \leq \frac{2 N}{K_0} C \cdot \frac{K_j^\beta}{|\log(\eps_j)|^{\rho}}.
   \end{align}
 \end{enumerate}
\end{lemma}

\begin{proof}
 By total ergodicity, in particular Lemma~\ref{lem:ergodic2}, we may find
 a set $\Omega_0 \subseteq \Omega$ such that
 $$
  \mu(\Omega_0) \geq 1 - \frac{1}{4}
 $$
 and for any $\omega\in\Omega_0$, we have that
 $\{V(n)\}_{n=0}^{N-1}$ is $(\delta,\sigma, \lceil\frac{N}{K}-1\rceil,\mathcal{E})$-critical for
 $N \geq N_0$ (some $N_0$). Similarly, we may find
 by Lemma~\ref{lem:ergodic2}
 for each $j \geq 1$ a set $\Omega_j$ such that
 $$
  \mu(\Omega_j) \geq 1 - \frac{1}{4} \frac{1}{2^j}
 $$
 and \eqref{eq:numresonant} holds for $N \geq N_j$. If we let
 $$
  \Omega_{\infty} = \bigcup_{j =0}^{\infty} \Omega_j,
 $$
 then we have that $\mu(\Omega_{\infty}) \geq \frac{1}{2}$.
 We will now fix $\omega \in \Omega_{\infty} \cap \Omega_{CS}$,
 where $\Omega_{CS}$ is as in Theorem~\ref{thm:cs}. This finishes
 the proof.
\end{proof}

In particular, we see that, we may choose $N/L = K(1 + o(1))$ in \eqref{eq:estiKj}.
We will now study the right hand side of \eqref{eq:numresonant}.

\begin{lemma}
 Assume \eqref{eq:aswegner1a}, \eqref{eq:3beta3rho}, and \eqref{eq:condtogetresonant}.
 Then \eqref{eq:numresonant} implies \eqref{eq:condtinyresonant}
 with $\delta = \delta_j$, $\sigma = \sigma_j$ and $\mathcal{E}$ as in
 \eqref{eq:defEJE}.
\end{lemma}

\begin{proof}
 The right hand side of \eqref{eq:condtinyresonant} satisfies
 $$
  \frac{\sigma_j}{4}(1 - 2 \sigma_j) \frac{L_j}{M_j+1} \geq \sigma \E^{-4\sigma} \E^{-\frac{1}{99}} L_0 10^{- 3(j+1)(j+2)}
 $$
 since $1 - 2 \sigma_j \geq \frac{1}{2}$, $j+4\leq -2(j+1)(j+4)$, \eqref{eq:productMj}, and \eqref{eq:boundsLj}.

 By \eqref{eq:lowbounddeltajsigmaj}, we have that $\sigma_j \delta_j \geq \sigma \delta 10^{j^2}$,
 and thus
 $$
  |\log(\eps_j)|^{\rho} \geq \left(\frac{\sigma\delta}{2}\right)^\rho 10^{\rho (j+1)(j+2)}.
 $$
 Combining this with \eqref{eq:estiKj},
 we obtain the following estimate for the right hand side of \eqref{eq:numresonant}
 $$
  \frac{2 N}{K_0} C \cdot \frac{K_j^\beta}{|\log(\eps_j)|^{\rho}}
  \leq 4 C \cdot L_0 \cdot \frac{\E^{4\beta\sigma} \E^{\frac{\beta}{99}} (2 K_0)^{\beta} 2^{\rho}}{\left(\sigma\delta\right)^\rho}
  \cdot 10^{-(\rho - 3 \beta) (j+1)(j+2)}.
 $$
 Now \eqref{eq:condtogetresonant} and \eqref{eq:3beta3rho} imply the claim.
\end{proof}

\begin{proposition}
 Assume \eqref{eq:3beta3rho} and \eqref{eq:condtogetresonant}.
 Then, for every $j\geq 1$ and $E$, there exists an $N_0 \geq 1$, such that
 $\{V_{\omega}\}_{n=0}^{N-1}$ is $(\delta_j,\sigma_j, L_j, [E - \eps_j, E + \eps_j])$-critical.
\end{proposition}

\begin{proof}
 By the last lemma, we can satisfy the conditions of
 Theorem~\ref{thm:multistep2} for all $i \leq j$, hence
 the claim follows.
\end{proof}

Now, we are ready for.

\begin{proof}[Proof of Theorem~\ref{thm:mainE}]
 Applying the last proposition for sufficiently large $j$, we see that we can
 satisfy \eqref{eq:cond3}, and by sufficiently large $N$, that we satisfy
 \eqref{eq:cond1}. Furthermore \eqref{eq:cond2} is automatically satisfied
 by our choice of $\eps_j$. Hence, we can apply Theorem~\ref{thm:multi1},
 to be in the same situation as discussed in Section~\ref{sec:proofthmmainB}.
 Applying the method of that section, we can conclude that there exists
 a set $\mathcal{E}_0 \subseteq \mathcal{E}$ of full measure, such that
 for every $E \in \mathcal{E}_0$, we have that
 $$
  L(E) \geq \E^{-8\sigma}\E^{-\frac{1}{99}} \gamma.
 $$
 We then even obtain the lower bound for every $E \in \mathcal{E}$
 by subharmonicity of $L(E)$. This finishes the proof.
\end{proof}

%
%

\section{The integrated density of states}

In this section, we quickly review some things about the integrated
density of states.
Let $(\Omega,\mu)$ be a probability space, $T: \Omega\to\Omega$
an ergodic transformation, and $f: \Omega\to\R$ a bounded real valued function.
We use the usual definition
\be
 V_{\omega}(n) = f(T^n\omega)
\ee
for $n\in \Z$ and $H(\omega)$ for the associated Schr\"odinger operator.
For $\Lambda\subseteq\Z$, we let $H_{\Lambda}(\omega)$ be the restriction
of $H(\omega)$ to $\ell^2(\Lambda)$. For some length scale $M \geq 1$, we introduce
\be\label{eq:defidsM}
 k_{M} (E) = \frac{1}{M} \int_{\Omega} \tr(P_{(-\infty,E)}(H_{[0,M-1]}(\omega))) d\mu(\omega).
\ee
We have the following lemma

\begin{lemma}\label{lem:idstomes}
 Assume
 $$
  k_M(E + \frac{\eps}{2}) - k_M(E - \frac{\eps}{2})
  \leq \frac{C M^{\beta}}{|\log(\eps)|^{\rho}}
 $$
 then
 \begin{align}\label{eq:conddistspec}
 \mu(\{\omega:\quad&\exists\Lambda\subseteq [0,M-1]:\quad \dist(E, \sigma(H_{\omega,\Lambda}))  \leq \frac{1}{2} \eps\})
   \leq \frac{C M^{2+\beta}}{|\log(\eps)|^{\rho}}.
 \end{align}
\end{lemma}

\begin{proof}
 For fixed interval $\Lambda \subseteq [0,M-1]$,
 and $\omega$, we have $\dist(E, \sigma(H_{\omega,\Lambda})) \leq \frac{1}{2} \eps$
 implies that
 $$
  \tr(P_{(-\infty,E + \frac{1}{2}\eps)}(H_{\omega,\Lambda}(\omega)))-
  \tr(P_{(-\infty,E - \frac{1}{2}\eps)}(H_{\omega,\Lambda}(\omega))) \geq 1.
 $$
 For $\Lambda = [a,b]$, we have $H_{\omega,\Lambda} = H_{T^{-a} \omega,[0, b-a-1]}$.
 So we see by \eqref{eq:defidsM} that with $n = \#\Lambda$
 $$
  \mu(\{\omega:\quad \dist(E, \sigma(H_{\omega,\Lambda})) \leq \frac{1}{2} \eps\})
   \leq k_{n}(E + \frac{1}{2}\eps) - k_{n}(E - \frac{1}{2}\eps) \leq \frac{C n^{\beta}}{|\log(\eps)|^{\rho}}.
 $$
 The claim follows by that there are less then $M$ subintervals of $[0,M-1]$
 with $n$ elements.
\end{proof}

We furthermore remark the following lemma, whose prove is an exercise
in elementary calculus.

\begin{lemma}\label{lem:loghoelder}
 Let $\alpha, \rho > 0$ and
 \be
  C(\alpha,\rho) = \E^{-\rho} \left(\frac{\rho}{\alpha}\right)^{\rho},
 \ee
 then for $0 < \eps < \frac{1}{2}$
 \be
  \eps^{\alpha} \leq \frac{C(\alpha,\rho)}{|\log(\eps)|^{\rho}}.
 \ee
\end{lemma}

%
%
%

\section{The integrated density of states for the skew-shift model}\label{sec:toy}

In this section, we will prove Proposition~\ref{prop:idsskew}. It turns out more
convenient to prove the following theorem.

\begin{theorem}\label{thm:idstoy}
 Let $\eps > 0$ and $N \geq 1$ an integer. Then
 \be
  k_{\lambda,N}(E + \eps) - k_{\lambda,N}(E) \leq
   7 \cdot \max(1,\frac{1}{\lambda}) \cdot N^2 \eps.
 \ee
\end{theorem}

Before proving this theorem, let us first derive Proposition~\ref{prop:idsskew}.

\begin{proof}[Proof of Proposition~\ref{prop:idsskew}]
 This follows by Lemma~\ref{lem:idstomes} and \ref{lem:loghoelder}.
\end{proof}

In order to prove Theorem~\ref{thm:idstoy}, we will need some
preparations. For $\delta > 0$ and $N \geq 1$, introduce the set
$\Omega(\delta,N)$ by
\be
 \Omega(\delta,N) = \{ \ul{\omega} \in \Omega:\quad
 (T^n \ul{\omega})_K \in [\delta, 1 - \delta],\quad 1 \leq n \leq N\}.
\ee
We have the following bound on the size of $\Omega(\delta,N)$.

\begin{lemma}
 We have
 \be
  |\Omega(\delta,N)| \geq 1 - 2 N \delta.
 \ee
\end{lemma}

\begin{proof}
 Let
 $$
  \Omega_b = \{\ul{\omega} \in \Omega:\quad\omega_K\notin [\delta,1-\delta]\}.
 $$
 We have that $|\Omega_b| = 2 \delta$.
 Observe that
 $$
  \Omega(\delta,N) = \Omega \backslash \bigcup_{n = 1}^{N} T^{-n} \Omega_b.
 $$
 The claim now follows by $T$ being measure preserving.
\end{proof}

We will need a bit of notation for $\ul{\omega} \in \Omega$,
we will denote by $\ul{\omega}' \in \mathbb{T}^{K-1}$ the first
$K-1$ components of $\ul{\omega}$, so
$$
 \ul{\omega} = (\ul{\omega}',\omega_K).
$$
We will show the following bound.

\begin{lemma}
 Given $\rho: \R\to [0,1]$ an increasing and differentiable function.
 The following bound holds
 \be\label{eq:boundintderi}
  \int_{\Omega(2\eps,N)} \frac{\partial}{\partial \omega_K} \tr(\rho(H_{\lambda,\ul{\omega},[1,N]} - t)) d\ul{\omega}
   \leq N+1.
 \ee
\end{lemma}

\begin{proof}
 We fix some $\ul{\omega}' \in \mathbb{T}^{K-1}$. We will let $\ul{\omega} = (\ul{\omega}', \vartheta)$,
 then $\frac{\partial}{\partial \omega_K}$ becomes $\frac{\partial}{\partial \vartheta}$.
 We have that the set
 $$
  A = \{\vartheta:\quad (\ul{\omega}',\vartheta) \in \Omega(2\eps,N)\}
 $$
 is some subset of $[0,1]$ consisting of at most $N + 1$ many intervals.
 So we may write
 $$
  A = [\vartheta_0, \vartheta_1] \cup [\vartheta_2, \vartheta_3] \cup \dots [\vartheta_{2N}, \vartheta_{2N+1}],
 $$
 For $0 \leq p \leq N$, we have that for $H_{\lambda,(\ul{\omega}',\vartheta),[1,N]}$
 and $H_{\lambda,(\ul{\omega}',\ti{\vartheta}),[1,N]}$ differ by a rank one perturbation
 for $\vartheta, \ti{\vartheta} \in [\vartheta_{2p}, \vartheta_{2p+1}]$. It is thus
 a standard fact, that
 \begin{align*}
  \int_{[\vartheta_{2p}, \vartheta_{2p+1}]} &\frac{\partial}{\partial \vartheta} \tr(\rho(H_{\lambda,(\ul{\omega}',\vartheta),[1,N]} - t)) d\vartheta\\
   & \leq \tr(\rho(H_{\lambda,(\ul{\omega}',\vartheta_{2p+1}),[1,N]} - t)) - \tr(\rho(H_{\lambda,(\ul{\omega}',\vartheta_{2p}),[1,N]} - t)) \leq 1
 \end{align*}
 By summing up, and integrating over $\ul{\omega}' \in \mathbb{T}^{K-1}$
 the claimed bound follows.
\end{proof}

Now, we come to

\begin{proof}[Proof of Theorem~\ref{thm:idstoy}]
 Let $\rho: \R \to \R$ be a smooth function such that
 $\rho(x) = 1$ for $x \leq 0$ and $\rho(x) = 0$ for $x\geq \eps$.
 We then observe that
 \begin{align*}
  \tr(P_{(-\infty, E+\eps)} H_{\lambda,\ul{\omega},[1,N]}) &- \tr(P_{(-\infty, E)}H_{\lambda,\ul{\omega},[1,N]})\\
   &\leq \tr(\rho(H_{\lambda,\ul{\omega},[1,N]} - E - \eps)) - \tr(\rho(H_{\lambda,\ul{\omega},[1,N]} - E + \eps)) \\
   & = \frac{1}{\lambda} \int_{E - \eps}^{E+\eps} \frac{\partial}{\partial t} \tr(\rho(H_{\lambda,\ul{\omega},[1,N]} - t)) dt.
 \end{align*}
 Since these functions are analytic, we can replace inside the set $\Omega(2\eps, N)$
 the $t$ derivate by a $\omega_K$ derivate. Hence, we obtain that
 \begin{align*}
  k_{\lambda,M}(E + \eps) &- k_{\lambda,M}(E) \\
  &\leq \max(1, \frac{1}{\lam}) \int_{\Omega(2\eps,N)} \int_{E - \eps}^{E+\eps} \frac{\partial}{\partial \omega_K} \tr(\rho(H_{\lambda,\ul{\omega},[1,N]} - t)) dt d\ul{\omega}\\
  &\quad + |\Omega\backslash \Omega(2\eps,N)| \cdot N,
 \end{align*}
 where we used the worst case estimate for $\ul{\omega} \notin \Omega(2\eps,N)$.
 The claim now follows by \eqref{eq:boundintderi}.
\end{proof}

\section*{Acknowledgments}

I am thankful to Daniel Lenz and G\"unter Stolz for their help, when I started learning
about these problems. I am indebted to David Damanik for useful advice and suggestions
and to Jon Chaika for useful conversations.

%
%
%

\end{document}